\documentclass[12pt]{amsart}
\usepackage{tikz,tikz-cd,
  fullpage,
  graphicx,amssymb,epic,eepic,cancel,amsthm,
  picins,
color,hyphenat,
  multicol, 
}
\usepackage[all]{xy}

\newtheorem*{theorem*}{Theorem}
\newtheorem{theorem}{Theorem}[section]
\newtheorem{lemma}[theorem]{Lemma}

\newtheorem{proposition}[theorem]{Proposition}
\newtheorem{corollary}[theorem]{Corollary}

\theoremstyle{definition}
\newtheorem{definition}{Definition}[section] 
\newtheorem{remark}{Remark}[section]
\newtheorem{example}[remark]{Example}

\def\Z{{\mathbb Z}}

\def\Q{{\mathbb Q}}

\def\calC{{\mathcal C}}
\def\calF{{\mathcal F}}

\def\C{{\mathbb C}}

\def\qdim{q{\rm dim \,}}

\def\Hom{\operatorname{Hom}}
\def\Ext{\operatorname{Ext}}

\def\dim{\operatorname{dim}}

\def\im{\operatorname{im}}

\def\id{{\mbox{id}}}

\def\G{{\Gamma}}

\def\-{\setminus}
\def\ld{{}^{\vee}\!}

\begin{document}

\newdimen\captionwidth\captionwidth=\hsize

\title{Koszul algebras and Flow lattices}

\author{Zsuzsanna Dancso and Anthony Licata}
\thanks{}
\address{Zsuzsanna Dancso,
  School of Mathematics and Statistics F07\\
  The Universit of Sydney\\
  Sydney NSW 2006\\
  Australia
}
\email{zsuzsanna.dancso@sydney.edu.au}
\urladdr{https://sydney.edu.au/science/people/zsuzsanna.dancso.php}
\address{Anthony Licata,
Mathematical Sciences Institute\\
Hanna Neumann Building\\
Australian National University\\
Acton ACT 2601\\
Australia
}
\email{amlicata@gmail.com}
\urladdr{https://maths-people.anu.edu.au/~licatat}

\date{\today}

\maketitle

\begin{abstract}
We provide a homological algebraic realization of the lattices of integer cuts and integer flows of graphs.
To a finite 2-edge-connected graph $\G$ with a spanning tree $T$, we associate a finite dimensional Koszul algebra $A_{\G,T}$.
Under the construction, planar dual graphs with dual spanning trees are associated Koszul dual algebras.
The Grothendieck group of the category of finitely-generated $A_{\G,T}$ modules is isomorphic to the Euclidean lattice $\Z^{E(\G)}$, and we
describe the sublattices of integer cuts and integer flows on $\G$ in terms of the representation theory of $A_{\G,T}$.  
The grading on $A_{\G,T}$ gives rise to $q$-analogs of the lattices of 
integer cuts and flows; these $q$-lattices depend non-trivially on the choice of spanning tree.  We give a $q$-analog of the matrix-tree theorem, and prove that the $q$-flow lattice of $(\G_1,T_1)$ is isomorphic to the $q$-flow lattice of $(\G_2,T_2)$ if and only if there is a cycle preserving bijection from the edges of $\G_1$ to the edges of $\G_2$ taking the spanning tree $T_1$ to the spanning tree $T_2$.  This gives a $q$-analog of a classical theorem of Caporaso-Viviani and Su-Wagner.
\end{abstract}

\tableofcontents

\section{Introduction}
\subsection{Lattices and Grothendieck groups}
An {\em integer lattice} is a finitely generated free Abelian group with a symmetric, $\Z$-valued bilinear form, sometimes required to be non-degenerate.  Integer lattices appear naturally in the representation theory of finite dimensional algebras:  if $A$ is a finite dimensional algebra, the Grothendieck group of the additive category of finitely-generated projective left $A$ modules, denoted $G_0$, is a free abelian group with basis the distinct isomorphism classes of indecomposable projective modules.  The pairing $([P],[Q]) = \dim\Hom(P,Q)$ is bilinear, and though it is not always symmetric, it sometime is -- e.g. when $A$ is a symmetric algebra.  In this way many important lattices in Lie theory arise in geometric representation theory as Grothendieck groups of additive categories.  

It is interesting to explore the relationship between properties of lattices and structures in homological algebra. To give an example, denote by $K_0$ the Grothendieck group of the abelian category of all finite-dimensional $A$-modules; $K_0$ is also a free abelian group, which by the Jordan-Holder theorem is spanned by the classes of the simple modules.  The $\Hom$ pairing identifies $K_0$ as the dual of $G_0$, at least as an abelian group.  If $G_0$ is a lattice, it is tempting to try to identify $K_0$ as the {\em dual lattice}, after endowing $K_0$ with the {\em Euler form}: $$\langle [M],[N]\rangle := \sum_{i=0}^\infty (-1)^i \dim \Ext^i_{A(B)} (M,N).$$ In general this alternating sum need not converge; however, if $A$ has finite homological dimension, the Euler form gives a bilinear pairing on $K_0$. Moreover, since in that case every simple module has a finite length projective resolution, the natural map $G_0 \longrightarrow K_0$ is an isomorphism of lattices, so that $G_0$ is a {\em unimodular lattice}. Thus, the lattice-theoretic notion of unimodularity is captured by a categorical notion of homological finiteness.

In lattice theory, unimodular lattices are often constructed by {\em gluing} non-unimodular lattices of smaller rank (see \cite[Chapters 4.3, 16]{CS}).  Perhaps the simplest example, which is the only one we consider in this paper, is when the unimodular lattice is the Euclidean lattice, and it is obtained by gluing together two smaller rank lattices.  This example is of interest in graph theory: given a 2-edge-connected graph $\G$ with edge set $E(\G)$, the lattice of integer flows $\calF(\G)$ and the lattice of integer cuts 
$\calC(\G)$ glue to form the Euclidean lattice $\Z^{E(\G)}$; in particular, $\calF(\G)$ and $\calC(\G)$ appear as mutual orthogonal complements inside the Euclidean lattice $\Z^{E(\G)}$.  

The main construction of the current paper gives a homological algebra lift of this particular occurrence of lattice gluing.  
More specifically, given a 2-connected graph $\G$ with a choice of a spanning tree $T$, we construct a {\em bipartite algebra} $A_{\G,T}$, as the quotient of a path algebra.  The first basic result is that 
$A_{\G,T}$ is standard Koszul; if follows from this that the Grothendieck groups $K_0\cong G_0$ equipped with the Euler/Hom pairing are isomorphic to the Euclidean lattice $\Z^{E(G)}$.  

Moreover, the subcategory of $A_{\G,T}\text{-mod}$ generated by a distinguished subset of projective modules descends in the Grothendieck group to the lattice of integer flows $\calF(\G)$; similarly, the lattice of integer cuts $\calC(G)$ is realized in the Grothendieck group as the span of a complementary collection of simple modules.  

When $\G$ is planar, the pair $(\G,T)$ has a planar dual $(\G^!,T^!)$, and it is immediate from the construction that the algebra assigned to $(\G^!,T^!)$ is the Koszul dual of that assigned to $(\G,T)$.  This is reminiscent of other combinatorial shadows of Koszul duality, such as the Gale duality/Koszul duality phenomena appearing in work of the second author with Braden--Proudfoot--Webster \cite{BLPW} on symplectic duality for hypertoric varieties.

From this point of view, the structure of the Euclidean lattice is categorified by that of a finite dimensional standard Koszul algebra.  It would be interesting to see whether or not other unimodular lattices, such as the Leech lattice and Niemeier lattices, appear as ``homologically finite" structures in classical or $p$-dg homological algebra.  Although we don't address it in the current document, this question is a large part of our motivation for writing this paper.

\subsection{$q$-lattices and graph theory}
One interesting piece of structure enjoyed by the bipartite algebras $A_{\G,T}$ is a non-negative (Koszul) grading.  We may therefore consider categories of graded 
$A_{\G,T}$ modules, so that the Grothendieck groups of these categories are free $\Z[q,q^{-1}]$ modules endowed with $\Z[q,q^{-1}]$-valued bilinear forms; these free modules are interesting in their own right.  

We define a {\em q-lattice} to be a finitely generated free $\Z[q,q^{-1}]$-module $L$ with a non-degenerate sesqui-linear (that is, linear in the second argument and $q$-anti-linear in the first argument) form
 $$
 	\langle \bullet,\bullet \rangle : L \times L \longrightarrow \Z[q,q^{-1}],
$$
together with a $q$-anti-linear involution $d : L \longrightarrow L$ such that for all $x,y\in L$,
$$
	\langle x,y \rangle = \langle d(y), d(x) \rangle,
$$
and such that $d$ becomes the identity map after setting $q=1$.  The Grothendieck groups of categories of graded $A_{\G,T}$ modules are then $q$-lattices.  We consider the analogs of basic lattice-theoretic notions such as duality, determinants, unimodularity, and gluing for $q$-lattices (Section~\ref{sec:qLattice}).  We define the $q$-cut and $q$-flow lattices of graphs, and, in Section~\ref{sec:CutFlow}, prove $q$-analogues of two famous classical results in graph theory.

We note here that are many generalizations of integer lattices in the literature of lattices to rings other than $\Z$, however, most focus on the theory of lattices over principal ideal domains.  Thus, as far as we are aware, there doesn't currently exist a well-developed theory of $q$-lattices, to which our $q$-cut and $q$-flow lattices belong. However, $\Z[q,q^{-1}]$ shares some features with principal ideal domains: for example, projective modules over $\Z[q,q^{-1}]$ are always free \cite{Swa}.   

It is easy to see from their construction that the isomorphism class of the $q$-cut and $q$-flow lattices of a graph-and-spanning-tree pair $(\G,T)$ 
depends non-trivially on the spanning tree $T$, in contrast to the theory of classical cut and flow lattices.  We now briefly explain the graph-theoretic information that is captured by including the parameter $q$.  

A theorem of Su--Wagner and Caporaso--Viviani \cite{SW, CV} states that for two (two-edge-connected) graphs $\G_1$ and $\G_2$, there exists a lattice isomorphism between the lattices of integer flows $\calF(\G_1) \cong \calF(\G_2)$, if and only if there exists a bijection $F: E(\G_1) \to E(\G_2)$, which preserves cycles. 
The $q$-analogue of this theorem, Theorem~\ref{thm:q-2-iso} of this paper, illustrates that the $q$-cut and $q$-flow lattices we define essentially determine the spanning tree $T$:
\begin{theorem*}
For two graphs $\G_1$ and $\G_2$ with respective spanning trees $T_1$ and $T_2$, there exists a q-lattice isomorphism
 $\calF_{q}(\G_{1},T_{1}) \cong \calF_{q}(\G_{2},T_{2})$, if and only if
there exists a cycle-preserving bijection of edges $F: E(\G_{1}) \to E(\G_{2})$ such that $F(T_{1})=T_{2}$.
\end{theorem*}

The classical Matrix-Tree Theorem (see for example \cite[Chapter 6]{Bi}) of graph theory, also known as the Kirchoff Theorem, states that the determinant of $\calC(\G)$ -- that is, the determinant of the pairing matrix of the bilinear form on $\calC(\G)$ -- counts the number of spanning trees of $\G$. In Theorem~\ref{thm:qMatrixTree} we prove the following $q$-analogue, which further illustrates how the $q$-cut and $q$-flow lattices retain the spanning tree information:

\begin{theorem*}
Let $\G$ be a graph with chosen spanning tree $T$.  Then the determinant of $\calC_q(\G,T)$ is a polynomial $\sum_{i=0}^r c_i q^{2i}$ with $c_0=1$, and $r=|E(T)|=|V(G)|-1$.  The coefficient $c_i$ is equal to the number of spanning trees of $\G$ which differ from $T$ in exactly $i$ edges.  
\end{theorem*}

Finally, we note that the constructions and results of this paper generalise in a straightforward way from graphs to regular matroids. Regular matroids possess a notion of duality (also known as Gale duality), which generalizes planar graph duality, and exchanges the lattices of integer cuts and integer flows. It is this more general matroid duality that really corresponds to Koszul duality at the level of bipartite algebras in our main construction. While we predominantly use the language of graphs for simplicity in the body of the paper, we will take short detours into the world of regular matroids to indicate how our constructions apply in that setting. A more matroid-centred approach can be found in the University of Sydney undergraduate honours thesis of Leo Jiang \cite{LJ}.

\subsection{Plan}The paper is structured as follows. In Section~\ref{sec:Algebra} we introduce bipartite algebras and establish basic properties of their module categories. In Section~\ref{sec:qLattice} we define $q$-lattices, and basic notions such as Gram matrices and lattice gluing.  Our guiding examples are Grothendieck groups of graded module categories over bipartite algebras. In Section~\ref{sec:CutFlow}, the most substantial section of this paper, we construct bipartite algebras from graphs with a choice of a spanning tree. We use their module categories to categorify the classical cut and flow lattices, study the $q$-cut and $q$-flow lattices obtained from this construction, and prove the graph theoretic theorems mentioned above.

\subsection{Acknowledgements}
We are grateful to Leo Jiang for his input in Section~\ref{sec:qLattice}, in particular the proof of statement~(1) of Theorem~\ref{thm:qGlue}, and for proofreading a draft of this paper. We thank Ben Webster for helpful conversations, and Tatiana Nagnibeda for pointing us to the valuable reference \cite{BHN} at an early stage of this project. 

Z.\ D. was funded by an Australian Research Council Discovery Early Career Research Award DE170101128. A.\ L.\ acknowledges support from the Australian Research Council Discovery Project 
DP180103150.

\section{Bipartite algebras}\label{sec:Algebra}

We start by introducing {\em bipartite algebras}, associated to bipartite graphs via a path algebra construction.
In Section~\ref{sec:CutFlow} we will discuss how bipartite algebras arise naturally from (not necessarily bipartite) graphs along with a 
choice of a spanning tree, and in this way they give rise to
a homological algebraic realization of the lattices of integer cuts and flows associated to graphs.

\begin{definition}\label{def:A(B)}
Let $B$ be a bipartite graph, with vertices $V(B) = V_0(B) \cup V_1(B)$.
Let $Q(B)$ denote the double quiver of $B$, that is, the directed graph in which each edge of $B$ is replaced by two arrows, one in each orientation.
We associate two $\C$-algebras to $B$:
\begin{itemize}
\item $A(B) := \text{Path}(Q(B))/\{\text{length two paths which start and end in } V_0(B)\}$,
\item $A^{!}(B) := \text{Path}(Q(B))/\{\text{length two paths which start and end in } V_1(B)\}$.
\end{itemize}
Here $\text{Path}(Q(B))$  denotes the {\em path algebra} of $Q(B)$, whose underlying vector space is spanned by the oriented paths in the quiver $Q(B)$. Multiplication is given by concatenation of paths whenever the last vertex of the first path agrees with the first vertex of the second, and defined to be zero otherwise. 

We call $A(B)$ the {\em bipartite algebra} associated to $B$. Both $A(B)$ and $A^{!}(B)$ are $\Z$-graded by path length.  We note also that it is immediate from the definitions that any path consisting of three or more edges is zero in both $A(B)$ and $A^{!}(B)$.  In what follows we denote by $e_i$ the constant path at vertex $i$: the $\{e_i\}$ are pairwise orthogonal idempotents, with 
$1 = \sum_{i\in V(B)} e_i$.  
\end{definition}

\subsection{Finitely generated graded $A(B)$-modules}\label{subsec:Modules}
Let $A(B)\text{-mod}$ denote the category of finitely generated {\em graded} left $A(B)$-modules.
Classifying simple and indecomposable projective $A(B)$-modules is straightforward from the quiver description of $A(B)$.  We state the basic results below, leaving proofs to the interested reader or to \cite{GG, LJ}.

\begin{proposition} Isomorphism classes of ungraded simple $A(B)$-modules are in bijection with the bipartite vertex set $V(B)$.
Denote the simple module corresponding to $e_i\in V(B)$ by $L_i$: as a vector space $L_i$ is one dimensional, $e_i$ acts as the identity, and all other $e_j$ as well as the positively graded subalgebra of $A(B)$, act by zero.
 \end{proposition}

As graded modules, the $L_i$ are declared to be contained in degree zero. All graded simple modules are shifts of the $L_i$.
 Since $A(B)$ is Artinian, indecomposable projective modules are in bijection with simple modules. Hence, we obtain the following:

\begin{corollary} The indecomposable projective $A(B)$-modules are given by the projective covers
of the simple modules. The projective cover $L_i$ is $P_i:=A(B) \cdot e_i$, which has a  as a vector space by paths which end at $i$.
\end{corollary}

The modules $P_i$ are naturally graded by path length, and all graded indecomposable projective $A(B)$-modules are shifts of the $P_i$.
Morphisms between indecomposable projective modules are given as follows:

\begin{proposition}\label{prop:HomP}
The set of $A$-homomorphisms $Hom(P_i,P_j)$ has a basis consisting of paths from $i$ to $j$.
 That is, $Hom(P_i,P_j)\cong e_iA(B)e_j$, as graded vector spaces.  The map which sends a path $\alpha$ to the reverse path $\overline{\alpha}$ extends linearly to a vector space isomorphism $\Hom(P_i,P_j)\cong \Hom(P_j,P_i)$. 
 \end{proposition}
 
Next, we show that $A(B)$ is a standard Koszul algebra. In general, a graded algebra $A$ is {\em Koszul} if every graded simple left module admits a finite linear 
projective resolution. Let $\{L_i: i\in I\}$ denote the simple left $A$-modules, and let $\{P_i: i\in I\}$ be their respective 
projective covers. We recall the definition of a standard Koszul algebra:

\begin{definition}\label{def:StdKoszul}
The Koszul algebra
$A$ is called {\em standard Koszul} if there exists a set of {\em standard modules} $\{V_i: i\in I\}$ with surjections
$P_i \xrightarrow{p_i} V_i \xrightarrow{\pi_i} L_i$, such that for some partial order $<$ on $I$,
\begin{enumerate}
 \item For all $ i\in I$ the module ker$(\pi_i)$ has a filtration where each sub-quotient is isomorphic to $L_j$ for some $j<i$, and
 \item For all $ i \in I$ the module ker$(p_i)$ has a filtration where each sub-quotient is isomorphic to $V_k$ for some $k>i$. 
\end{enumerate}
\end{definition}

In this case, the category $A(B)$-mod is called a {\em highest weight category}. For a more detailed introduction to 
 highest weight categories and Koszul algebras, see Section~5 of \cite{BLPW}. Algebras whose module categories are
 highest weight are also called {\em quasi-hereditary}.

\begin{proposition}\label{prop:ProjRes}
 Each simple graded left $A(B)$-module $L_i$ has a finite linear projective resolution. 
\end{proposition}

\begin{proof}
 Let $N_i$ denote the neighbourhood of the vertex $i$ in $B$, that is, the set of vertices adjacent to $i$. Let
 $x_{ij}$ denote the length one path $i \rightarrow j$, where $i$ and $j$ are adjacent vertices.
 We'll denote the degree shift in the path length grading by $\{\cdot\}$.
 
 First assume $i\in V_1(B)$. Then the linear projective resolution for $L_i$ is
 $$\bigoplus_{j\in N_i} P_j \{1\}  \xrightarrow{\sum \cdot x_{ji}} P_i \to L_i.$$
 
 Similarly, if $i \in V_0(B)$, then $L_i$ has linear projective resolution
 $$\bigoplus_{j\in N_i} \bigoplus_{k \in N_j} P_k \{2\} \xrightarrow{\sum \cdot x_{kj}} 
 \bigoplus_{j\in N_i} P_j \{1\}  \xrightarrow{ \sum \cdot x_{ji}} P_i \to L_i.$$
 Note that in the sum $\bigoplus_{j\in N_i} \bigoplus_{k \in N_j} P_k \{2\}$, each indecomposable projective module
 may appear multiple times.
 \end{proof}

\begin{proposition} \label{prop:KoszulDual} With respect to the partial order given by $V(1)<V(0)$,  
the standard modules over $A(B)$ are $V_i=P_i$ when $i\in V(0)$, and $V_i=L_i$ when $i \in V(1)$.
\end{proposition}

\begin{proof}
We need to check that the conditions (1) and (2) of Definition~\ref{def:StdKoszul} are satisfied.
 Assume that $i \in V(0)$. In this case $V_i=P_i$, so condition $(2)$ is vacuous. As for condition (1),
 a basis for ker$(\pi_i)$ consists of the length one paths $x_{ji}$ for $j \in N(i)$. Let $N(i)=\{{j_1},...,{j_r}\}$.
 Then the filtration $A(B)x_{j_1i} \subseteq A(B)x_{j_1i}\oplus A(B)x_{j_2i}\subseteq ... \subseteq A(B)x_{j_1i}\oplus A(B)x_{j_2i}\oplus...\oplus A(B)x_{j_ri}$
 satisfies condition (1).
 The proof is similar for $i \in V(1)$.
\end{proof}

 A Koszul algebra $A$ is always quadratic, and in this case the quadratic dual algebra is also called the Koszul dual, and denoted by $A^!$.

\begin{proposition}\label{prop:quadratic}
 The algebra $A^{!}(B)$ of Definition \ref{def:A(B)} is the Koszul dual of $A(B)$.
\end{proposition}

\begin{proof} Let $R=\C\{e_i: i\in V(B)\}$ be the ring generated by the idempotents $e_i$, and
$M=\C\{x_{ij}: x_{ij} \in E(Q(B))\}$ be the $R$-bimodule spanned by the edges of $Q(B)$. Then the tensor algebra $T:=T_R(M)$ is the
path algebra of $Q(B)$. 

We have $A(B)=T/TWT$, where $W \subseteq M\otimes_R M$ is the $R$-bimodule of length two paths which start and end in $V(0)$, as in Definition \ref{def:A(B)}. 
The annihilator $W^\perp \subseteq M^* \otimes M^*$ 
is easily identified with the set of length two paths which start and end in $V(1)$,
which shows that the quadratic dual $T_R(M^*)/T_R(M^*)W^\perp T_R(M^*)$ is isomorphic to $A^!(B)$. 
\end{proof}

Let $L_i$ be a simple $A(B)$-module and $P_i$ its projective cover. The {\em injective hull} $I_i$ of $L_i$, as a vector space, is isomorphic to the linear dual
of $P_i$, with the negative path length grading. This is naturally an $A(B)^{op}$-module. There is an isomorphism $A(B)\xrightarrow{\cong} A(B)^{op}$ sending a path
$x$ to the reverse path $\overline{x}$. This defines the action of $A(B)$ on $I_i=P_i^*$.

Similarly, costandard modules $\Lambda_i$ are the linear duals of the corresponding
standard modules $V_i$, with the action of $A(B)$ defined via the bar isomorphism as above. This in particular means that
$\Lambda_i\cong L_i$ for $e_i\in V_{1}(B)$, and $\Lambda_i\cong I_i$ for $e_i\in V_{0}(B)$.

\begin{remark}\label{rmk:d}
Note that ``acting on the linear dual via the bar isomorphism'' is a contravariant auto-equivalence of {\em ungraded} finitely generated $A(B)$-modules,
denoted $d$, with the properties that $dL_i=L_i$, $dP_i=I_i$ and $dV_i=\Lambda_i$. However, $d$ reverses gradings and grading shifts, for instance $d(M\{1\})= (dM)\{-1\}$,
and similarly for morphisms. \qed
\end{remark} 

\subsection{The Grothendieck group}\label{subsec:K0}
The category $A(B)\operatorname{-mod}$ of finite dimensional graded $A(B)$-modules is an abelian category. The {\em graded Grothendieck group} 
$K_0:=K_0(A(B)\operatorname{-mod})$ is, by definition, the $\Z[q,q^{-1}]$-module generated by the isomorphism classes of modules,
modulo the relations $[A]-[B]+[C]=0$ for any short exact sequence of modules $0\!\to\! A\!\to\! B\!\to\! C\!\to\! 0$; and $q[A]=[A\{1\}]$, where
$\{\cdot\}$ denotes the grading shift.

\begin{definition}\label{def:EulerForm} The {\em graded Euler form} on $K_0$ is defined by
$$\langle [M],[N]\rangle := \sum_{i=0}^\infty (-1)^i \qdim \Ext^i_{A(B)} (M,N).$$
Here $\qdim$ denotes the graded dimension, that is, $\qdim \Ext^i_{A(B)} (M,N)$ is the Laurent polynomial in which the coefficient of $q^k$ 
is the dimension of the degree $k$ piece of $\Ext^i_{A(B)} (M,N)$.
\end{definition}

The graded Euler form is a non-degenerate $q$-sesqui-linear form on $K_0(A(B) \operatorname{-mod})$.
Here $q$-sesqui-linear means that for any $f\in \Z[q,q^{-1}]$ and $M,N \in A(B)\operatorname{-mod}$, with $\bar{f}(q):=f(q^{-1})$,
$$\langle \bar{f}[M],[N]\rangle={f}\langle [M],[N]\rangle=\langle [M],f[N]\rangle.$$ 
Sesquilinearity is a reflection of the behaviour of the $\Hom$ functor with respect to grading shifts: 
$$\Hom(M\{-1\}, N)=\Hom\{1\}(M,N)=\Hom(M,N\{1\}).$$

\begin{proposition}
$K_0$ is a free $\Z[q,q^{-1}]$-module of rank $n=|V(B)|$.
The classes of simple modules $\{[L_i]\}_{i\in V(B)}$, indecomposable projectives $\{[P_i]\}$, indecomposable injectives $\{[I_i]\}$, standard modules $\{[V_i]\}$, and costandard modules $\{[\Lambda_i]\}$ all form bases for $K_0$. 
\end{proposition}
\begin{proof}
 The existence of a Jordan-Holder filtration implies that the classes of simple modules form a spanning set for $K_0$. The uniqueness of the 
 Jordan-Holder composition factors means that the classes of simple modules indeed form a basis, and the rank of $K_0$ is $n$.
 Every simple module has a finite projective resolution, hence the isomorphism classes of indecomposable projective modules also span, and since there is $n$ of them they form a basis.

 For any quasi-hereditary algebra $A$, costandard modules form a right dual set to standard modules in $A$-mod \cite[3.11]{CPS}
 meaning that 
 $$\Ext^i(V_k,\Lambda_l)=
 \begin{cases}
  \C \text{ if }k=l\text{ and }i=0, \\
  0 \text{ otherwise.}
 \end{cases}
$$
Hence, $\langle [V_k], [\Lambda_l]  \rangle= \delta_{kl}$. This implies that standard and costandard modules are both independent
sets. To see this, assume for example that there is a linear dependence $\sum_{k=1}^n f_k(q) [V_k]=0$, where $f_k(q) \in \Z[q,q^{-1}]$. Then
$$0=\Big\langle \sum_k f_k(q) [V_k], \sum_{l=1}^n f_l(q) [\Lambda_l]\Big\rangle = \sum_{k=1}^n f_k(q^{-1})f_k(q).$$
Note that $f_k(q^{-1})f_k(q)$ has a non-negative constant term, which is zero only if $f_k=0$. Hence, $f_k=0$ for all $k$, a contradiction.
Therefore, standard and costandard modules each form bases.

The involution $d$ of Remark~\ref{rmk:d} induces a $q$-antilinear map on $K_0$ with $d[L_i]=[L_i]$. So, if 
$[L_k]=\sum_{i=1}^n c_i(q)[P_i]$ with $c_i(q) \in \Z[q,q^{-1}]$, then $[L_k]= d[L_k]=\sum_{i=1}^n c_i(q^{-1}) [I_i]$. Hence the indecomposable injective modules $\{[I_i]\}$ also span and hence form a basis.
\end{proof}

\begin{remark} The involution $d$  satisfies a symmetry property with respect to the graded Euler form:
$\langle [M],[N] \rangle =\langle d[N], d[M] \rangle.$ To check this, write $[M]$ and $[N]$ as $\Z[q,q^{-1}]$-linear combinations of simple modules.
\end{remark}

\section{q-Lattices}\label{sec:qLattice}
Recall that a {\em lattice} is a finitely generated free abelian group with a symmetric bilinear form, typically required to be non-degenerate. The following
definition lifts lattices to free modules over the ring of Laurent polynomials: the guiding example is the graded Grothendieck group of Section~\ref{sec:Algebra}. In this section we establish the analogues of some basic notions and facts of lattice theory. 

\begin{definition}
 A {\em q-lattice} is a finitely generated free $\Z[q,q^{-1}]$-module $L$ with a non-degenerate sesquilinear\footnote{That is, $q$-anti-linear in the first argument, and linear in the second argument, as described after Definition~\ref{def:EulerForm}.} form
 $$
 	\langle \bullet,\bullet \rangle : L \times L \longrightarrow \Z[q,q^{-1}];
$$
and equipped with a $q$-anti-linear involution $d : L \longrightarrow L$ such that for all $x,y\in L$
$$
	\langle x,y \rangle = \langle d(y),d(x) \rangle,
$$
and 
$
	d_{q=1}: L\otimes_{\Z[q,q^{-1}]}\Z \longrightarrow L\otimes_{\Z[q,q^{-1}]} \Z \text{ is the identity map.}
$
\end{definition}

After setting $q=1$, a $q$-lattice becomes an ordinary lattice, although note that it is possible for a non-degenerate $q$-pairing to become degenerate at $q=1$. The involution $d$ deforming the identity is used to give the appropriate $q$-analog for the symmetry of the form.  

Since the sesquilinear form is not symmetric, one has to distinguish between various ``left'' and ``right'' notions: left and right duals, orthogonality and complements will be defined later in this section. The involution $d$ can be used to move between opposite side notions.  

\begin{example}\label{ex:K0qLattice}
In this language, the Grothendieck group $K_{0}$ of the category of finite-dimensional graded $A(B)$-modules is a $q$-lattice under the graded Euler form, with $d$ induced by the duality map defined in Remark~\ref{rmk:d}. 
On $K_0$, $d$ is the unique $q$-anti-linear map which fixes the classes $[L_{i}]$ of simple modules. 
\end{example}

\subsection{Gram matrices and change of basis} Recall that given a basis \linebreak $\mathcal B =\{b_1,...,b_n\}$
for a classical $\Z$-lattice $L$,
the bilinear form is encoded in the {\em Gram matrix} $A = A_{\mathcal B}$, where $A_{ij}=\langle b_i,b_j\rangle \in \Z$.
Given any two elements $x,y\in L$, let $x_{\mathcal B}$ and $y_{\mathcal B}$ denote the column vectors expressing $x$ and $y$, 
respectively, in the basis $\mathcal B$. Then $\langle x,y \rangle= x_{\mathcal B}^t Ay_{\mathcal B}$, where the superscript $t$
denotes the matrix transpose. 

Let $\mathcal B' =\{b'_1,...,b'_n\}$ be a different basis for $L$, and $T=T_{\mathcal B}^{\mathcal B'}$ the change of basis matrix.
That is, the $i$-th column of $T$ is the vector $(b'_i)_\mathcal B$. Then the Gram matrix $A_{\mathcal B'}$ is $T^tA_{\mathcal B}T$.
The Gram matrix determines the lattice up to isomorphism, and two lattices with Gram matrices $A$ and $A'$ are isomorphic if and 
only if there exists an integer matrix $T$, invertible over $\Z$, such that $A'=T^tAT$.

Similarly, given a choice of a basis $\mathcal B =\{b_1,...,b_n\}$
for a $q$-lattice $L$,
the sesquilinear form is again encoded in the Gram matrix $A_{\mathcal B}$, where $A_{ij}=\langle b_i,b_j\rangle \in \Z[q,q^{-1}]$.
Given elements $x,y\in L$, and $x_{\mathcal B}$ and $y_{\mathcal B}$ the column vectors expressing $x$ and $y$ 
in the basis $\mathcal B$, we have $\langle x,y \rangle= x_{\mathcal B}^* Ay_{\mathcal B}$, where the superscript $*$
denotes the matrix transpose composed with the involution $q \mapsto q^{-1}$. 

Let $\mathcal B' =\{b'_1,...,b'_n\}$ be another basis of $L$, and again let $T$ denote the change of basis matrix, 
where the $i$-th column of $T$ is the vector $(b'_i)_\mathcal B$. Note that $T$ is an invertible matrix over $\Z[q,q^{-1}]$, 
so that $\det T=\pm q^k$ for some $k\in\Z$. Then the Gram matrix with respect to the basis $\mathcal B'$ is $T^{*}AT$.
It is still true that the Gram matrix determines the $q$-lattice up to isomorphism, and two 
$q$-lattices given by Gram matrices $A$ and $A'$ are isomorphic if and 
only if there exists an invertible matrix $T$ over $\Z[q,q^{-1}]$ such that $A'=T^{*}AT$.

Note in particular that the determinant of the Gram matrix is an isomorphism invariant of both classical lattices and $q$-lattices, that is, it is independent of the choice of basis in which the Gram matrix is written. Hence, it is also called the {\em determinant of the lattice}.

\subsection{Dual $q$-lattices}
In classical lattice theory, given an integer lattice $L$, the {\em lattice dual} is defined as $L^{\vee}:=\{x\in L\otimes_{\Z}\Q \, | \,  \langle x, v \rangle \in \Z \quad \forall v\in L\}$. The lattice dual is typically not an integer lattice, as elements may pair non-integrally with each other. The lattice dual is a free abelian group of the same rank as $L$, with a $\Q$-valued symmetric bilinear form, and includes $L$.

Given a basis  $\mathcal B =\{b_1,...,b_n\}$ of the classical lattice $L$, the dual basis in $L\otimes_{Z}\Q$ is the unique basis $\mathcal B^{\vee} =\{b_1^{\vee},...,b_n^{\vee}\}$ with the property that $\langle b_{i},b_{j}^{\vee}\rangle=\delta_{ij}$ for all $1\leq i,j \leq n$. The dual basis $\mathcal B^{\vee}$ is a basis for $L^{\vee}$. 

If $L=L^{\vee}$,  the lattice is called {\em unimodular}.  A lattice is unimodular if and only if its determinant, the determinant of the Gram matrix, is a unit of $\Z$, that is, $\pm 1$.  The Gram matrix $A_\mathcal B$ also arises as the change of basis matrix $T_{\mathcal B^{\vee}}^{\mathcal B}$, or in other words the matrix of the inclusion map $L \hookrightarrow L^{\vee}$, with respect to the bases $\mathcal B$ and $\mathcal B^{\vee}$. 

We now describe the analogous duality notions for $q$-lattices.

\begin{definition}
Given a $q$-lattice $L$, we define the {\em right dual} $L^{\vee}$ of $L$ to be 
 $$L^{\vee}:=\{x\in L\otimes_{\Z[q,q^{-1}]}\Q(q) \, | \,  \langle v,x \rangle \in \Z[q,q^{-1}] \quad \forall v\in L\}.$$
Given a basis  $\mathcal B =\{b_1,...,b_n\}$ of $L$, the {\em right dual basis} in $L\otimes_{Z[q,q^{-1}]}\Q(q)$ is the unique basis $\mathcal B^{\vee} =\{b_1^{\vee},...,b_n^{\vee}\}$ with the property that $\langle b_{i},b_{j}^{\vee}\rangle=\delta_{ij}$ for all $1\leq i,j \leq n$. 

Similarly, the {\em left dual} of $L$ is $\ld L:=\{x\in L\otimes_{\Z[q,q^{-1}]}\Q(q) \, | \,  \langle x,v \rangle \in \Z[q,q^{-1}] \quad \forall v\in L\}$,
and the {\em left dual basis} of $\mathcal B$ is the unique basis $\ld\mathcal B =\{\ld b_1,...,\ld b_n\}$ with the property that $\langle \ld b_{i},b_{j}\rangle=\delta_{ij}$ for all $1\leq i,j \leq n$. Note that $d$ intertwines left/right duality in the sense that $(d(\mathcal B))^\vee = d ( \ld \mathcal B)$.

\end{definition}

\begin{remark} The left and right dual $q$-lattices are indeed free $\Z[q,q^{{-1}}]$-modules, with basis given by the left and right dual bases, respectively.  They are not necessarily $q$-lattices, as the pairing on them might not be valued in the Laurent polynomial ring $\Z[q,q^{-1}]$.
As for the existence and uniqueness of the dual bases, observe that given a non-degenerate $q$-sesquilinar pairing, one can orthogonalise bases in a $\Q(q)$-vector space using the Gram-Schmidt process (without normalisation). This implies that left and right orthogonal complements exist and are unique. Given a basis vector $b_{i} \in \mathcal B$, the dual basis vector $\ld b_{i}$ (or $b_{i}^{\vee}$) lies in the one-dimensional left (or right) orthogonal complement of the subspace spanned by $\mathcal B \setminus \{b_{i}\}$, and can be chosen so that $\langle \ld b_{i}, b_{i}\rangle=1$, (or $\langle b_{i},b_{i}^\vee \rangle=1$).
\end{remark}

For example, in $K_0:=K_0(A(B)\operatorname{-mod})$, costandard modules form a right dual basis to standard modules, indecomposable projectives are left dual to the simple modules, and indecomposable injectives
are right dual to the simple modules. This, as well as the asymmetry of the orthogonality relation, is illustrated in the following example.

\begin{example}\label{ex:LRDualBases}
 Let $B$ denote the complete bipartite graph on three vertices with $V_0(B)=\{e_1, e_2\}$ and $V_1(B)={e_3}$.
 Then the Grothendieck group $K_0$ of finitely generated graded $A(B)$-modules is a rank three $q$-lattice. The
 isomorphism classes of indecomposable projective modules $[P_i]$, for $i=1,2,3$, form a basis, and with respect
 to this basis the Gram matrix is symmetric:
 $$
 \begin{bmatrix}
  1 & 0 & q \\
  0 & 1 & q \\
  q & q & 1+2q^2 \\
 \end{bmatrix}
 $$
 The basis $\{[P_1], [P_2], [P_3]\}$ is left dual to the basis given by the simple modules $\{[L_1], [L_2], [L_3]\}$,
so that $\langle [P_i], [L_j]  \rangle= \delta_{ij}$. However, $\langle[L_j], [P_i] \rangle$ is not necessarily equal to $\delta_{ji}$:
 from Proposition~\ref{prop:ProjRes} we see that
 $$[L_1]=(1+q^2)[P_1]+q^2[P_2]-q[P_3],$$ 
 $$[L_2]=q^2[P_1]+(1+q^2)[P_2]-q[P_3],$$ 
 $$[L_3]=-q[P_1]-q[P_2]+[P_3].$$
 Hence, for example, $\langle [L_1],[P_1] \rangle=q^{-2}$, and $\langle [L_3],[P_1] \rangle=q-q^{-1}$. Note that at the value $q=1$ one recovers the symmetric ungraded
 Euler form, so that $\langle [L_j],[P_i] \rangle_{q=1}=\delta_{ji}$.
 
 The standard modules in this example are $V_1=P_1$, $V_2=P_2$, and $V_3=L_3$; the costandard modules are $\Lambda_1=I_1$, $\Lambda_2=I_2$ and $\Lambda_3=L_3$.
 The classes of the costandard modules form a right dual basis to the classes of the standard modules. Note that in the ungraded case, when $q=1$, the involution
 $d$ is the identity on $K_0$, and hence $[V_i]_{q=1}=[\Lambda_i]_{q=1}$. That is, in the ungraded case both the standard and the costandard modules form
 orthonormal bases for $K_0$ at $q=1$. For general $q$, however, the Gram matrix in the basis $\{V_1,V_2,V_3\}$ is:
 $$
 \begin{bmatrix}
  1 & 0 & 0 \\
  0 & 1 & 0 \\
  q-q^{-1} & q-q^{-1} & 1 \\
 \end{bmatrix}
 $$
 This pattern holds more generally for the pairings between graded standard modules for bipartite algebras: $\langle [V_i],[V_j] \rangle=\delta_{ij}$ unless $i \in V_1(B)$, $j\in V_0(B)$ and $i$ is adjacent to $j$ in $B$, in which case $\langle [V_i],[V_j] \rangle=q-q^{-1}$.
\end{example}

\begin{proposition}\label{prop:GramChangeBasis}
 If $L$ is a $q$-lattice with $\mathcal B$ a basis, then the Gram matrix $A_{\mathcal B}$ coincides with the matrix of
the embedding $L \hookrightarrow L^{\vee}$ with respect to the bases $\mathcal B$ and $\mathcal B^{\vee}$. The matrix of the embedding $L \hookrightarrow \ld L$ with respect to the bases $\mathcal B$ and $\ld \mathcal B$ coincides with the conjugate $\overline{A}_{\mathcal B}$, that is, the involution $q \mapsto q^{-1}$ applied to each matrix entry of $A_{\mathcal B}$.
\end{proposition} 

\begin{proof}
Like the classical case, this is elementary linear algebra. We need to show that $b_{i}= \sum_{j=1}^{n} \langle b_{j},b_{i} \rangle b_{j}^{\vee}$, and $b_{i}= \sum_{j=1}^{n}\overline{\langle b_{i}, b_{j}\rangle} \,\, \ld b_{j}$, for each $i=1,...,n$. For the first equality, take the pairing on the left with each basis vector $b_{k}$; for the second, do the same but on the right.
\end{proof}

\begin{example}\label{ex:AllMatrices}
As before, let $K_0$ denote the Grothendieck group of finitely generated graded $A(B)$-modules, a $q$-lattice. Let $G$ denote the Gram matrix in the basis given by the classes of indecomposable projective modules $\{[P_i]\}$. Simple modules form a right dual basis to this, so Proposition~\ref{prop:GramChangeBasis} means that the $i$-th column of $G$ is the vector $[P_i]$ written in the basis $\{[L_i]\}$. Applying $d$, one obtains that the matrix whose $i$-th column is the class of the injective module $[I_i]$, written in the basis $\{[L_i]\}$, is $\overline{G}$, where the bar denotes the involution $q\mapsto q^{-1}$.
In turn, via changes of bases, this implies that the Gram matrix in the basis $\{[L_i]\}$ is $\overline{G}^{-1}$, and the Gram matrix in the basis $\{[I_i]\}$ coincides with $G$.
\end{example}

\begin{definition}
A $q$-lattice is {\em unimodular} if $L^{\vee}=L$.
\end{definition}

Luckily, we don't need to distinguish between left and right unimodularity: even without the assumption of the existence of the involution $d$, it is true that $L^{\vee}=L$ if and only if $\ld L =L$. To see this, observe that if $L^{\vee}=L$ for a $q$-lattice $L$ with basis $\mathcal B$, then $\mathcal B^{\vee}$ is also a basis for $L$. Since, by definition, $\ld(\mathcal B^{\vee}) = \mathcal B$, it follows that $\ld L =L$. Observe also that a $q$-lattice is unimodular if and only if the determinant (of the Gram matrix) is a unit in $\Z[q,q^{-1}]$, that is, $\det L=\pm q^{k}$.

\subsection{$q$-Lattice gluing}
In classical lattice theory, {\em lattice gluing} is an important construction which allows for producing larger, indecomposable lattices from smaller components, see for example \cite[Chapter 4.3]{CS}. To glue two lattices $L_{1}$ and $L_{2}$, one takes their direct sum, and then adjoins carefully selected elements of $L_{1}^{\vee}\oplus L_{2}^{\vee}$ to obtain an integer lattice. Of course the construction can be generalised to more than two glued components. 

Lattice gluing is of particular interest when the resulting lattice is {\em unimodular}, and used in the classification of unimodular lattices of small rank \cite[Chapter 16]{CS}. It is therefore important to understand when it may be possible to glue two lattices together so that the end result is unimodular. In this case, the original lattices are embedded as mutual orthogonal complements in the unimodular lattice. A theorem (this formulation due to \cite{BHN}, also in \cite{CS}) provides an important necessary condition for when this is possible:

\begin{theorem}\label{thm:ClassGlue}
\cite{BHN} Let $L$ be a unimodular lattice, and $L_{1}$ and $L_{2}$ sub-lattices which are mutual orthogonal complements of each other within $L$. Then
\begin{enumerate}
\item The images of the orthogonal projections of $L$ onto $L_{1} \otimes_\Z \Q$ and $L_{2}\otimes_\Z \Q$ are $L_{1}^{\vee}$ and $L_{2}^{\vee}$, respectively.
\item The {\em glue groups} of $L_{1}$ and $L_{2}$ are isomorphic, that is, $L_{1}^{\vee}/L_{1}\cong L_{2}^{\vee}/L_{2}$.
\item The determinants of $L_{1}$ and $L_{2}$ are equal.
\end{enumerate}
\end{theorem}

We end this section with the $q$-analogue of this theorem:

\begin{theorem}\label{thm:qGlue}
Let $L$ be a unimodular $q$-lattice and let $L_{1}$ and $L_{2}$  be sub-lattices such that $L/L_{1}$ and $L/L_{2}$ are free $\Z[q,q^{-1}]$-modules.  Assume further that $L_{1}^{\perp} =L_{2}$ ($L_{2}$ is the right orthogonal complement of $L_{1}$ in $L$), and $L_{1}=^{\perp}\!\! L_{2}$ ($L_{1}$ is the left orthogonal complement of $L_{2}$ in $L$).
Then:
\begin{enumerate}
\item Any $x\in L$ can be uniquely written as a sum $x=x_{1}+x_{2}$ where $x_{1}\in L_{1}\otimes_{\Z[q,q^{-1}]} \Q(q)$ and $x_{2}\in L_{2}\otimes_{\Z[q,q^{-1}]}  \Q(q)$. This defines left/right orthogonal projections 
$$\pi_{1}: L \to L_{1}\otimes_{\Z[q,q^{-1}]}\Q(q), \ \ \pi_{1}(x)=x_{1}$$  and 
$$\pi_{2}: L \to L_{2}\otimes_{\Z[q,q^{-1}]}\Q(q),\ \ \pi_{2}(x)=x_{2}.$$ Then the image of $\pi_{1}$ is $L_{1}^{\vee}$ and the image of $\pi_{2}$ is $\ld L_{2}$.
\item $L_{1}^{\vee}/L_{1} \cong \ld L_{2}/L_{2}$.
\item The determinants of $L_{1}$ and $L_{2}$ are equal up to units.
\end{enumerate}
\end{theorem}

\begin{proof}
We prove the first statement for the projection $\pi_{1}$. Since $L/L_{1}$ is free, every basis $\mathcal B_{1}=\{b_{1},...,b_{r}\}$ of $L_{1}$ can be extended to a basis $\mathcal B=\{b_{1},...,b_{r}, c_{r+1},...,c_{n}\}$ of $L$. 

Since $L$ is unimodular, the right dual basis $\mathcal \mathcal B^{\vee}=\{ b_{1}^{\vee},...,b_{r}^{\vee}, c_{r+1}^{\vee},...,c_{n}^{\vee}\}$ is also a basis for $L$. Note that for $j > r$, $c_{j}^{\vee} \in (L_{1}\otimes_{\Z[q,q^{-1}]} \Q(q))^{\perp}=L_{2}\otimes_{\Z[q,q^{-1}]}  \Q(q)$. Hence, $\pi_{1}(c_{j}^{\vee})=0$, and therefore $\{\pi_{1}(b_{j}^{\vee}): j\leq r\} $ spans the image of $\pi_{1}$. 

For $j\leq r$, $b_{j}^{\vee}=\pi_{1}(b_{j}^{\vee})+\pi_{2}(b_{j}^{\vee})$. 
Then, for any $k\leq j$, $\langle  b_{k}, \pi_{2}(b_{j}^{\vee})\rangle =0$, since $\pi_{2}(b_{j}^{\vee})\in (L_{1}\otimes \Q(q))^{\perp}$. So  $\langle b_{k}, \pi_{1}(b_{j}^{\vee}) \rangle= \langle b_{k}, b_{j}^{\vee} \rangle= \delta_{jk}$, so $\{\pi_{1}(b_{j}^{\vee}): j\leq r\}$ is a right dual basis to $\mathcal B_{1}$. This completes the proof of statement (1) for $\pi_{1}$, and the case of $\pi_{2}$ is similar.

For the second statement, we will prove that $L_{1}^{\vee}/L_{1} \cong L/ (L_{1}\oplus L_{2}) \cong \ld L_{2}/L_{2}$. Consider the map of $\Z[q,q^{-1}]$-modules $\overline{\pi_{1}}: L/(L_{1}\oplus L_{2}) \to L_{1}^{\vee}/L_{1}$, given by $[x]\mapsto [\pi_{1}(x)]$, where $x \in L$, and square brackets denote cosets.
The map $\overline{\pi_{1}}$ is well defined since $\pi_{1}(L_{1}\oplus L_{2})=L_{1}$. 

It is also injective: if $[\pi_{1}(x)]=0$, then $\pi_{1}(x)\in L_{1}$.
Now, for all $x\in L$, $x = \pi_1(x) + \pi_2(x)$, with $\pi_2(x)\in \ld L_2$.  But $L_{2}$ is the right orthogonal complement of $L_{1}$ in $L$, so if $x \in L$ and $\pi_1(x)\in L_1$, then $\pi_2(x)\in L_2$.  Thus $[\pi_1(x)]=0$ implies $x \in L_{1}\oplus L_{2}$, hence $\overline{\pi_1}$ is injective. In addition, $\overline{\pi_{1}}$ is surjective because the image of $\pi_{1}$ is $L_{1}^{\vee}$. This completes the proof of the first isomorphism, and the second is similar.

For statement (3), choose bases $\mathcal B_{1}=\{b_{1},...,b_{r}\}$ for $L_{1}$ and $\mathcal B_{2}= \{b_{r+1},...,b_{n}\}$ for $L_{2}$. By part (1), the image of $\pi_{1}$ is $L_{1}^{\vee}$, and the kernel of $\pi_{1}$ is $L_{2}$. Furthermore, by assumption $L/L_{2}$ is free. So, given the dual basis $\mathcal B_{1}^{\vee}=\{b_{1}^{\vee},..., b_{r}^{\vee}\}$, we can obtain a basis of $L$ by choosing arbitrary vectors $\pi_{1}^{-1}(b_{j}^{\vee})$ in the preimage of $b_j^\vee$ and taking the union 
$\pi_{1}^{-1}(\mathcal B_{1}^{\vee})\cup \mathcal B_{2}$, where $\pi_{1}^{-1}(\mathcal B_{1}^{\vee})=\{\pi_{1}^{-1}(b_{j}^{\vee}): j\leq r\}$.
Similarly, after choosing vectors $\pi_{2}^{-1}(\ld b_{k})$, we have that $\mathcal B_{1}\cup \pi_{2}^{-1}(\ld \mathcal B_{2})$ also forms a basis of $L$, where 
$\pi_{2}^{-1}(\ld \mathcal B_{2})=\{\pi_{2}^{-1}(\ld b_{k}): r<k\leq n\}$.

Consider the $q$-lattice dilations $D_{1}$ and $D_{2}$,
\[
\begin{tikzcd}
L \arrow{r}{D_{1}} & L_{1}\oplus L_{2} 
\\
&L\arrow{u}{D_{2}} 
\end{tikzcd}
\]
where $D_{1}$ is defined by taking the ordered basis vectors $\pi_{1}^{-1}(\mathcal B_{1}^{\vee})\cup \mathcal B_{2}$ to the ordered basis $\mathcal B_{1} \cup \mathcal B_{2}$ ; similarly $D_{2}$ maps $\mathcal B_{1}\cup \pi_{2}^{-1}(\ld \mathcal B_{2})$ to $\mathcal B_{1} \cup \mathcal B_{2}$.
After extending scalars to $\Q(q)$, note that now $L\otimes \Q(q)=L_{1}\otimes \Q(q) \oplus L_{2}\otimes \Q(q)$, and $D_{1}\otimes \id$ and $D_{2}\otimes \id$ are vector space automorphisms. (We abuse notation and write $D_{1}\otimes \id$ and $D_{2}\otimes \id$ as simply $D_{1}$ and $D_{2}$):
\[
\begin{tikzcd}
L\otimes_{\Z[q,q^{-1}]}  \Q(q) \arrow{r}{D_{1}} & (L_{1}\otimes_{\Z[q,q^{-1}]}  \Q(q)) \oplus (L_{2}\otimes_{\Z[q,q^{-1}]}  \Q(q) )
\\
&L\otimes_{\Z[q,q^{-1}]}  \Q(q) \arrow{u}{D_{2}} \arrow{lu}{D_{1}^{-1}\circ D_{2} }
\end{tikzcd}
\]
In particular, by composing $D_{2}$ with  $D_{1}^{-1}$ we obtain an automorphism of $L\otimes \Q(q)$, which restricts to a lattice automorphism of $L$. The determinant of a $q$-lattice automorphism is a unit in $\Z[q,q^{-1}]$, that is, the determinant of $D_1^{-1}\circ D_2$ is equal to $\pm q^{k}$ for some integer $k$.
Therefore, the determinants of $D_{1}$ and $D_{2}$ agree up to units.

What remains is to compute these determinants. We accomplish this by writing $D_{1}$, as a matrix in the basis 
$\mathcal B_{1}^{\vee}\cup \mathcal B_{2}= \{b_{1}^{\vee},...,b_{r}^{\vee},b_{r+1},...,b_{n}\}$. We know that $D_{1}$ is the identity when restricted to $L_{2}\otimes_{\Z[q,q^{-1}]} \Q(q)$. Furthermore, for $1\leq i \leq r$, we have $b_{i}^{\vee}=\pi_{1}^{-1}(b_{i}^{\vee})+\beta_{i}$, for some $\beta_{i} \in L_{2}\otimes_{\Z[q,q^{-1}]}  \Q(q)$, and so $D_{1}(b_{i}^{\vee})=b_{i}+\beta_{i}$. 

Proposition~\ref{prop:GramChangeBasis} implies that the matrix of the linear map $$L_{1}\otimes_{\Z[q,q^{-1}]}  \Q(q) \to L_{1}\otimes_{\Z[q,q^{-1}]}  \Q(q)$$
 $$b_{i}^{\vee} \mapsto b_{i},$$ written in the basis $\mathcal B_{1}^{\vee}$, is the Gram matrix 
$A_{L_{1},\mathcal B_{1} }$ of $L_{1}$.  So the matrix of $D_{1}$ is a block diagonal matrix
$
\left[
\begin{array}{c|c}
A _{L_{1},\mathcal B_{1}} & \quad 0\quad \\ \hline
 * & \quad\id\quad
\end{array}\right].
$

Therefore, the determinant of $D_{1}$ is equal to the determinant of $L_{1}$. Similarly, the determinant of $D_{2}$ is equal to the determinant of $L_{2}$, completing the proof.

\end{proof}

\begin{remark}
We note here that the proof of the above theorem did not use the existence of the involution $d:L\longrightarrow L$ deforming the identity.  
Note also that in the classical Theorem~\ref{thm:ClassGlue}, statement (3) is almost immediate from (2): the glue groups are finite abelian groups, and the determinant is the size of the glue group by a geometric counting argument. In Theorem~\ref{thm:qGlue}, the statements remain true but the proof is more involved, as the left and right glue groups are torsion $\Z[q,q^{-1}]$-modules. 

In Theorem~\ref{thm:ClassGlue} the mutual orthogonal complement condition implies that the abelian group $L/L_{i}$ is torsion free for $i=1,2$, which in turn implies that these abelian groups are free. In Theorem~\ref{thm:qGlue}, the mutual orthogonal complement assumption implies that the $\Z[q,q^{-1}]$-modules $L/L_{i}$ are torsion free, but it does not follow formally from this that they are free modules.  This is why we assume they are free in the statement of Theorem~\ref{thm:qGlue} but not in Theorem~\ref{thm:ClassGlue}.
\end{remark}

\section{The $q$-lattices of integer cuts and flows associated to a spanning tree}\label{sec:CutFlow}
\subsection{Classical cut and flow lattices}
Let $\Gamma$ be a finite graph, with loop edges and multiple edges allowed. 
For simplicity, we assume that $\Gamma$ is 2-edge-connected, that is, $\Gamma$ is connected and remains connected 
after the removal of any one edge. In this paper we'll abbreviate this and say that $\Gamma$ is 2-connected. 

Let $V=V(\Gamma)$ and $E=E(\Gamma)$ denote the vertex set and the edge set of $\Gamma$, respectively. 
Choose an arbitrary orientation $\omega$ for $\Gamma$ (that is, a direction for each edge): this makes $\Gamma$ into a 1-dimensional CW-complex with
cellular chain complex
$$0 \to C_1(\Gamma, \Q) \stackrel{\partial}{\longrightarrow} C_0(\Gamma, \Q) \to 0,$$
where for an edge $e$ that begins at vertex $v$ and ends at $w$, $\partial(e)=w-v$. We equip $C_1(\Gamma, \Q)$ and  $C_0(\Gamma, \Q)$
with inner products $\langle \cdot, \cdot \rangle$ by declaring $E$ and $V$ to be orthonormal bases. Then the adjoint of the boundary map, $\partial^*: C_0(\Gamma, \Q) \to C_1(\Gamma, \Q)$
can be expressed by the formula $\partial^*(v)=\sum_{e \in E} \langle \partial(e),v \rangle  e$. 

If $|E|=n$, then $C_1(\Gamma,\Q)\cong\Q^n$ has an integer lattice $C_1(\Gamma,\Z)\cong\Z^n$ inside it. In $C_1(\Gamma,\Z)$, the sublattices 
$\{\im(\partial^*)\cap C_1(\Gamma, \Z)\}$ and  $\{\ker(\partial) \cap C_1(\Gamma, \Z)\}$ are mutual orthogonal complements. 
They are called the \emph{lattice of integer cuts} (the \emph{cut lattice}) and the \emph{lattice of integer flows} (the \emph{flow lattice}),
respectively, and each inherits the inner product restricted from $C_1(\Gamma, \Z)\cong\Z^n$. We denote the cut lattice and the flow lattice by $\calC(\Gamma)$ and $\calF(\Gamma)$, respectively. We denote the inner products by $\langle \cdot , \cdot \rangle_{{\calC(\Gamma)}}$ and $\langle \cdot , \cdot \rangle_{{\calF(\Gamma)}}$ respectively, dropping the subscripts whenever it does not cause confusion.

Note that even though the definition of $\calC(\Gamma)$ and $\calF(\Gamma)$ depends on the choice of an orientation, changing the orientation does not
change the isomorphism type of the lattices. Let $\omega'$ be another orientation of $\G$ that differs from $\omega$ by switching the direction of a single edge $e$, and let $\calC'(\Gamma)$ and $\calF'(\Gamma)$ be the cut- and flow lattices corresponding to the orientation $\omega'$. The isomorphism 
$C_{1}(\Gamma) \xrightarrow{\cong} C_{1}(\G)$,
which sends $e$ to $-e$, and all other edges to themselves induces isomorphisms $\calC(\G)\cong \calC'(\G)$ and $\calF(\G)\cong \calF'(\G)$.

In combinatorial terms, the cut lattice is generated by the \emph{cuts} of $\G$, as follows. Given a partition $V=V_0\cup V_1$, the corresponding {\em cut}
is the signed sum of the edges connecting $V_0$ to $V_1$, where each edge oriented from $V_{0}$ towards $V_{1}$ participates with positive sign, and edges oriented from $V_{1}$ towards $V_{0}$ participate with negative sign. 

The flow lattice, in turn, is generated by 
oriented \emph{cycles} (that is, closed walks) in $\Gamma$. Each oriented cycle gives rise to an element of $\calF(\G)$: a signed sum of the edges of the cycle, in which edges  participate with positive sign if their orientation agrees with the orientation of the cycle, and with negative sign otherwise. 

This combinatorial description leads to a (well-known) construction of bases for $\calC(\G)$ and $\calF(\G)$.
Fix a spanning tree $T$ of $\G$. Removing an edge $e \in T$ splits $T$ into two connected components. 
This defines a vertex partition $V=V_{0,e} \cup V_{1,e}$, where $e$ is directed from $V_{0,e}$
towards $V_{1,e}$. Denote the corresponding cut by $$K_e=\sum_{e'\in E_{V_0\to V_1}} e' -\sum_{e''\in E_{V_1\to V_0}} e''.$$ 
This is called the {\em fundamental cut} corresponding to the edge $e \in T$. Note that the only 
spanning tree edge appearing in $K_e$ is $e$ and it always appears with positive sign. The set of fundamental cuts $\{K_e: e\in T\}$ forms a 
basis for $\calC(\G)$. 

As for the flow lattice, every edge $f \notin T$ creates a single cycle when added to $T$, called the \emph{fundamental cycle} of $f$, and denoted $C_f$.
The corresponding basis element of $\calF(\G)$ is $$C_f=\sum_{f'\in C_f} \pm f'.$$ Here, the signs in the sum are assigned as follows:
the orientation of $f$ in $\omega$ induces
an orientation of the cycle $C_f$, and each edge $f'\in C_f$ appears with a positive sign if its orientation agrees with the orientation of $C_f$,
and with a negative sign otherwise. The set of fundamental cycles $\{C_f: f \in T^c= E\setminus T\}$ forms a basis of $\calF(\G)$.

\subsection{Bipartite algebras from graphs}\label{subsec:AlgebrasFromGraphs}
Let $\Gamma$ be a 2-connected graph with numbered edge set $E=\{1,2,...,n\}$ and orientation $\omega$, and fix a spanning tree $T\subset E$. 
Define a bipartite graph $B_{\Gamma,T}$ whose {\em vertex set} is the edge set of $\G$, with the bipartition
$$
	V_0(B_{\Gamma,T}) = T,  \ \ V_1(B_{\Gamma,T}) = E\setminus T.
$$
The edges of $B_{\Gamma,T}$ are defined as follows. 
Connect the vertex $j \in V_{1}(B_{{\Gamma,T}})$ to exactly those vertices ${i}\in V_{0}(B_{\Gamma,T})$ which occur in the 
fundamental cycle $C_{j}$ of $j$, as shown on an example in Figure \ref{fig:BipartiteGraph}.

\begin{figure}
 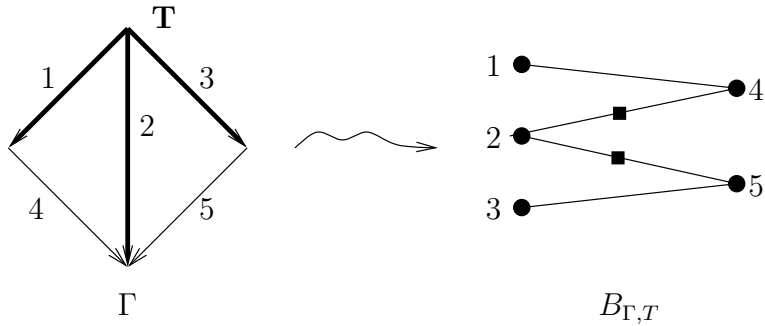
 \caption{The bipartite graph $B_{\G,T}$ associated to a graph. Spanning tree edges of $\G$ are drawn bold, negative edges of $B_{\G,T}$
 are marked with a square.}
 \label{fig:BipartiteGraph}
\end{figure}

Next, we decorate the edges of $B_{\G,T}$ with signs arising from the orientation $\omega$.
Consider an edge $f$ in the bipartite graph $B_{\Gamma,T}$.  This edge corresponds to a pair of edges $(i,j)$ in $\Gamma$, 
with $i\in T$, ${j}\notin T$. The edge $i$ participates in $C_{j}$ with a positive or negative sign:
this is the edge sign associated to $f$. In Figure \ref{fig:BipartiteGraph} we denote negative edges of $B_{\Gamma,T}$ by a small
square placed on the edge.

Note that the bases for $\calC(\G)$ and $\calF(\G)$ associated to the spanning tree $T$ can be read
off $B_{\G,T}$. It is a short combinatorial exercise to show that the edge $i$ appears in the fundamental cycle $C_j$, if and only if $j$
appears in the fundamental cut $K_i$. Furthermore, their signs are always opposite: $i$ appears in $C_j$ with positive sign if and only if $j$ appears in $K_i$
with negative sign. 

While changing the orientation $\omega$ of $\G$ does not change the isomorphism types of $\calC(\G)$ and $\calF(\G)$, it has an effect on the edge signs of $B_{\G,T}$. Namely, changing the orientation of an edge $i$ of $\G$ flips
the sign of all the edges of $B_{\G,T}$ incident to the vertex $i$, but it does not otherwise change the bipartite graph. 

Set $A_{\Gamma,T} := A(B_{\Gamma,T})$, as defined in Section~\ref{sec:Algebra}.
The edge signs of $B(\G,T)$ turn the double quiver $Q(B(\G,T))$ into a signed quiver (both directions of an edge inherit the same sign).
This induces a $\Z_2$-grading on the path algebra: the $\Z_{2}$-degree of a path is the parity of the number of negative edges in it. Since the relations are homogeneous, there is an induced $\Z_{2}$-grading on $A_{\G,T}$, and on finitely generated $A_{\G,T}$-modules.  The Groethendieck group $K_0(A_{\G,T}\text{-mod})$ is now a free module over the ring $\Z[q,q^{-1},t]/(t^2=1)$, and
the graded Euler form is valued in $\Z[q,q^{-1},t]/(t^2=1)$, and is $t$-bilinear and $q$-sesquilinear. The substitutions $t=1$ and $t=-1$ both turn the Grothendieck 
group into a $q$-lattice. From here on we will also refer to the $\Z_{2}$-grading as the $t$-grading.

The statements of
Section~\ref{sec:Algebra} hold with $t=1$, and otherwise are easily modified to include the $t$-grading. The simple modules $L_{i}$ are contained in $t$-degree 0. The indecomposable projective modules  $P_{i}$ are naturally $t$-graded, as their elements can be viewed as paths. Homomorphisms between indecomposable projective modules are spanned by paths as in Proposition~\ref{prop:HomP}, and $t$-graded accordingly. In 
Proposition~\ref{prop:ProjRes}, projective resolutions can be adjusted to account for the $t$-grading by inserting $t$-shifts whenever the edge ``leading to'' an indecomposable projective module has 
a negative sign. For instance, in the example of Figure~\ref{fig:BipartiteGraph}, the projective resolution of $L_4$ is $P_2\{1\}\langle1\rangle \oplus P_1\{1\} \to P_4$, where $\langle \cdot \rangle$ denotes the $t$-grading shift. The functor $d$ respects the $t$-grading, and descends to a $t$-linear, $q$-anti-linear map on the Grothendieck group.

\begin{proposition}\label{prop:FlowCategorification} For edges $i, j \notin T$ the pairing
$
	\langle [P_{i}], [P_{j}] \rangle|_{q=1,t=-1}
$
agrees with the pairing $\langle C_{i},C_{j}\rangle_{\calF(\Gamma)}$ of the corresponding fundamental cycles in the flow lattice of $\Gamma$.
\end{proposition}

\begin{proof}
Since $P_{i}$ and $P_{j}$ are projective modules, the $\Ext$  in the definition of the graded Euler pairing reduces to $\Hom$. So we need to show that when $i,j\notin T$, $$qt\!\dim \Hom(P_{i},P_{j})|_{q=1,t=-1}= \langle C_{i},C_{j}\rangle_{\calF(\Gamma)}.$$  

As shown in Proposition~\ref{prop:HomP}, a basis for $\Hom(P_{i},P_{j})$ is given by paths from $e_{i}$ to $e_{j}$ in $B_{\Gamma,T}$ all of which are of length 2, except for the length zero path $e_{i}$ in the case $i=1$. The length 2 paths, in turn, correspond to $T$-edges in common between $C_{i}$ and $C_{j}$. The substitution $t=-1$ gives these $T$-edges signs according to whether an edge participates in $C_{i}$ and $C_{j}$ with the same sign or opposite signs. This is exactly how the pairing $\langle C_{i},C_{j}\rangle_{\calF(\Gamma)}$ is defined.
\end{proof}

\begin{remark}\label{rmk:FlowInjectives}
In light of Example~\ref{ex:AllMatrices}, we may substitute indecomposable injective modules for the
indecomposable projective modules in the statement of Proposition~\ref{prop:FlowCategorification}. 
Specifically, for any signed bipartite graph $B$ and modules $M, N \in A(B)\text{-mod}$, in $K_{0}(A(B)\text{-mod})$ we have 
$\langle[M],[N] \rangle(q,t)=\langle [dN], [dM] \rangle(q^{{-1}},t)$. Hence,  for $e_{i}, e_{j}\notin T$, and indecomposable injective $A_{\G,T}$-modules  
$I_{i}$ and $I_{j}$, 
$\langle [I_{i}], [I_{j}] \rangle|_{q=1,t=-1}=\langle C_{i},C_{j}\rangle_{\calF(\Gamma)}$.
\end{remark}

The following is a dual statement to Proposition~\ref{prop:FlowCategorification}, which we will prove in Section~\ref{subsec:Duality}:

\begin{proposition}\label{prop:CutCategorification}
For edges $i, j \in T$ the pairing
$
	\langle [L_{i}], [L_{j}] \rangle |_{q=1,t=-1}
$
agrees with the pairing $\langle K_{i},K_{j}\rangle_{\calC(\Gamma)}$ of the corresponding fundamental cuts in the cut lattice of $\Gamma$.

\end{proposition}

\subsection{Koszul duality and matroid duality}\label{subsec:Duality} Rather than proving Proposition~\ref{prop:CutCategorification} directly we explain how it is dual to the statement of Proposition~\ref{prop:FlowCategorification}. Proposition~\ref{prop:KoszulDual} implies that if $B^{!}$ denotes the bipartite graph obtained from $B$ by switching the bipartition, then $A(B^{!})\cong A^{!}(B)$, where $A^{!}(B)$ is the Koszul dual of $A(B)$. 

If $B$ is a signed bipartite graph, we set the convention that $B^!$ is the bipartite graph obtained from $B$ by switching the bipartition and {\em all edge signs}. This induces a $t$-grading on $A(B^{!})\cong A^{!}(B)$.

\subsubsection{Koszul duality and Grothendieck groups}

By Theorem 1.2.6 of \cite{BGS}, there is a Koszul duality derived equivalence between {\em derived categories of graded modules}
$$K: \mathcal D^{b}(A(B)\text{-mod})\to \mathcal D^{b}(A(B^{!})\text{-mod}).$$ If $M$ is a module over $A(B)$, we write $M\in \mathcal D^{b}(A(B)\text{-mod})$ and mean the complex whose homological degree 0 chain module is $M$, and whose other chain modules are zero; and similarly for modules over $A(B^!)$. 

If $M \in A(B)\text{-mod},$ $\{\cdot\}$ denotes the internal grading shift, and $[\cdot]$ denotes the homological grading shift, then $K(M\{n\})=(KM)[-n]\{-n\}$. On the other hand, $K$ commutes with $t$-shifts.

For $i=1,...,n$, $K$ 
sends the simple module $L_{i}$ over $A_{\G,T}$ to the indecomposable projective module $P^{!}_{i}$ over the Koszul dual $A_{\G,T}^!$. 

The Grothendieck group of $\mathcal D^{b}(A(B)\text{-mod})$ coincides with $K_0(A(B)\text{-mod})$ via the graded Euler characteristic, that is, the alternating sum of graded dimensions of chain modules.
Koszul duality induces a $\Z$-linear map
$$K: K_{0}(A(B)\text{-mod}) \to K_{0}(A(B^{!})\text{-mod}),$$ determined by $K[L_{i}]=[P^{!}_{i}]$, $K(q)=-q^{-1}$, and $K(t)=t$. 
Furthermore, we have $$\langle [M], [N] \rangle = \langle K(M), K(N) \rangle|_{q\to -q^{-1}}.$$

\subsubsection{Planar Graph Duality and Matroid Duality}If $\Gamma$ is an oriented planar graph with spanning tree $T$, and $\Gamma^{!}$ its oriented planar dual, then there is a canonical bijection $\varphi: E(\Gamma)\to E(\G^{!})$, and 
$E(\G^{!}) \setminus \varphi(T)$ forms a spanning tree for $\G^{!}$. It is a short exercise to check that the bipartite graph $B_{\G^{!},E(G)\setminus T}$ coincides with $(B_{\G,T})^!$. See Figure~\ref{fig:PlanarDual} for an example.

\begin{figure}
\includegraphics[width=13.5cm]{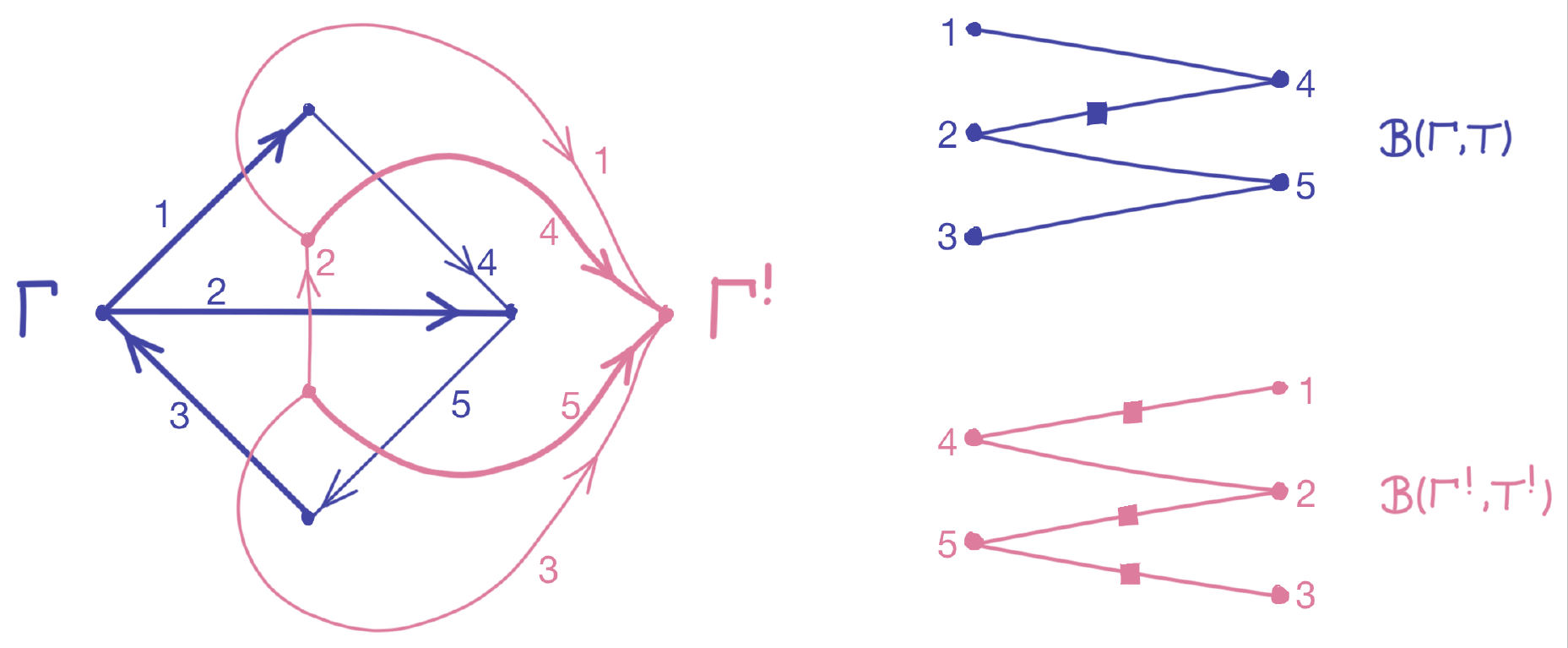}
\caption{On the left, we show an oriented graph $\G$ with spanning tree $T$ shown in thick edges, and its oriented planar dual $\G^{!}$, with spanning tree $T^{!}$ complementary to $T$. The corresponding signed bipartite graphs are shown on the right: $B_{\G^{!},T^{!}}$ is obtained from $B_{\G,T}$ by flipping the bipartition and all edge signs. As an example, the fundamental cycle $C_{2}^{!}$ and the fundamental cut $K_{2}$ are both given by ``$2+4-5$''.}\label{fig:PlanarDual}
\end{figure}

Furthermore, $\calF(\G)\cong \calC(\G^{!})$ canonically, and vice versa. For $i\in T$, $K_{i}=C_{i}^{!}$, where $K_i$ denotes the fundamental cut for $i$ in $\calC(\G)$, and $C_i^!$ denotes the fundamental cycle for $i$ in $\calF(\G^!)$. For $i\notin T$, $C_{i}=K_{i}^{!}$. 

\vspace{2mm}
{\em Proof of Proposition~\ref{prop:CutCategorification} for planar graphs.} From Proposition~\ref{prop:FlowCategorification} for the dual graph $\G^{!}$, 
we have that for $\G$-edges $i,j\in T$, $$\langle [P_{i}^!], [P_{j}^!] \rangle|_{q=1,t=-1}=\langle C_{i}^!,C_{j}^!\rangle_{\calF(\Gamma^!)}.$$
On the left, Koszul duality gives $$\langle [L_{i}], [L_{j}] \rangle|_{q=1,t=-1}=\langle [P_{i}^!], [P_{j}^!] \rangle|_{q=-1,t=-1}=\langle [P_{i}^!], [P_{j}^!] \rangle|_{q=1,t=-1},$$
since pairings between the classes of ``same side'' indecomposable projective modules are functions of $q^2$.

On the right, planar graph duality gives 
$$
	\langle C_{i}^!,C_{j}^!\rangle_{\calF(\Gamma^!)} = \langle K_{i},K_{j}\rangle_{\calC(\Gamma)},$$
and combining the two sides one obtains the statement of Proposition~\ref{prop:CutCategorification}.\qed

\vspace{2mm}

If $\G$ is not planar, there is no dual graph, though there is still a Koszul dual bipartite algebra, and a dual (flipped) bipartite graph. In this case these dualities correspond to the {\em matroid dual} of the oriented graphical matroid associated to $\G$. Thus in fact the most natural language for describing the combinatorial origin for Koszul duality of bipartite algebras is to assign a bipartite graph and bipartite algebra to any {\em oriented regular matroid} with a matroid basis, rather than just to a graph with a spanning tree. 
To simplify language in this paper, we will not present the details of the construction for regular matroids, and instead define the lattices of integer cuts and flows at the bipartite graph level. However, for the reader who is interested in matroid statements, we give a brief summary of the basic points here.

Each oriented regular matroid has a regular {\em oriented matroid dual}. Each oriented graph has an associated oriented {\em graphic matroid}, which is always regular, and for planar graphs the dual matroid of the graphic matroid agrees with the graphic matroid of the planar dual graph.
For non-planar graphs, the dual of the graphic matroid is a regular, non-graphic matroid. 

Regular matroids have associated lattices of integer cuts and flows, and everything in this paper can be extended to that context without difficulty. In particular, the analogue of a spanning tree is a {\em basis} of the matroid, and the analogue of a spanning tree basis for $\calC(\G)$ or $\calF(\G)$ is called the {\em fundamental basis} of the cut or flow lattice {\em associated to a basis of the matroid}. 
The complement of a basis for an oriented regular matroid forms a basis for the dual matroid, and matroid duality exchanges cut and flow lattices. Hence, the planar graph statements above also apply to regular matroids.

\subsubsection{Cuts and Flows for Bipartite Graphs} Given a signed bipartite graph $B$ with numbered vertex set $E=E_{0}\cup E_{1}$, let $\Z^{E}$ denote the Euclidean lattice generated by elements of $E$ (ie, elements of $E$ form an orthonormal basis). The flow lattice $\calF(B)$ of $B$ is the sublattice of $\Z^{E}$ generated by fundamental cycles $C_{i}$ associated to each $i \in E_{1}$ as follows: $$C_{i}=i+\sum_{j\in N_{i}} \epsilon_{ij} j,$$ where $N_{i}$ 
denotes the set of neighbours of $i$ and $\epsilon_{ij}$ is the sign of the edge $i$---$j$ in B. Similarly, the cut lattice
$\calC(B)$ of $B$ is the sublattice of $\Z^{E}$ spanned by fundamental cuts $K_{i}$ for $i\in E_{0}$, where 
$$K_{i}=i-\sum_{j\in N_{i}}\epsilon_{ij} j.$$

It is clear that for a bipartite graph arising from a graph with a spanning tree, $\calF(B_{\G,T})\cong \calF(\G)$ and $\calC(B_{\G,T})\cong \calC(\G)$, and the isomorphisms preserve fundamental cuts and flows. Furthermore, $\calF(B)\cong \calC(B^{!})$
and $\calC(B) \cong \calF(B^{!})$. Proposition~\ref{prop:FlowCategorification} and its proof remains true for flow lattices of signed bipartite graphs.

{\em Proof of Proposition~\ref{prop:CutCategorification}.} The proof for planar graphs generalises directly, replacing $T$ with $E_{0}$, and planar graph duality with bipartite graph duality.\qed

\subsubsection{Bipartite graphs and classes of matroids}
For those interested in matroid aspects, we clarify the relationship between graphs, signed bipartite graphs, and various classes of matroids.

From an oriented graph $\G$ with a choice of spanning tree $T$, we constructed a signed bipartite graph.
From the pair $(\G, T)$, one also obtains an oriented graphic matroid $\mathcal M(G)$ with a chosen basis $\mathcal B$.  
It is possible to construct the signed bipartite graph $B_{\G,T}$ from $\mathcal M(G)$: the vertices correspond to the base set $E$ of the matroid, partitioned into the basis and non-basis elements. Signed edges incident to a non-basis element are drawn according to the fundamental signed circuit in the fundamental basis of $\calF(\mathcal M(G))$ corresponding to $\mathcal B$. In fact, this construction of signed bipartite graphs from matroids works not only for graphic matroids but for any oriented regular matroid with a chosen basis.

Conversely, for any signed bipartite graph $B$, 
we can construct an (oriented) matroid with a chosen basis. The base set of this matroid is the vertex set $E$ of $B$, the basis is the $E_{1}$ side of the vertex partition, and circuits are generated by the fundamental circuits corresponding to elements of $E_{1}$. This matroid is always $\Q$-representable, but not necessarily regular. For example, the complete bipartite graph on $2+2$ vertices, with one negative edge sign, gives rise to a matroid not representable over $\mathbb F_{2}$.

In general, a $\Q$-representable matroid with a chosen basis would give rise to a bipartite graph with $\Z$-weighted edges as opposed to a signed bipartite graph. So signed bipartite graphs can be seen as a class of matroids between regular matroids and $\Q$-representable matroids.

\subsection{The $q$-cut and $q$-flow lattices}\label{subsec:qCutFlow}
For a signed bipartite graph $B$ with vertex set $E=E_{0}\cup E_{1}$, the Groethendieck group $K_{0}(A(B)\text{-mod})|_{t=-1}$ is a $q$-lattice in the sense of Section~\ref{sec:qLattice}, with the graded Euler form at $t=-1$, and the involution $d$ of Remark~\ref{rmk:d}. We define the $q$-cut and $q$-flow lattices as the appropriate $q$-sublattices of the Grothendieck group.  After giving the definition, we will see that it is also possible to define these objects combinatorially, without any reference to the bipartite algebra $A(B)$ or its representation category. The rest of this Section focuses on the intricate combinatorial properties of these new graph invariants.

\begin{definition}
For a signed bipartite graph $B$, with vertex set $E=E_{0}\cup E_{1}$, the {\em $q$-flow lattice} $\calF_{q}(B)$ is the $\Z[q,q^{-1}]$-submodule of $K_{0}=K_{0}(A_{B}\text{-mod})|_{t=-1}$ generated by 
the classes of projective modules $\{[P_{i}]: i \in E_{1}\}$, with the inherited sesquilinear form. The involution $d$ is defined to be the $\Z$-linear map which fixes the classes $\{[P_{i}]\}$, and for which $d(q)=q^{-1}$.

Similarly, the {\em $q$-cut lattice} $\calC_{q(B)}$ is generated by the classes of simple modules
 $\{[L_{i}]: i\in E_{0}\}$, with the inherited form, and $d([L_{i}])=L_{i}$.

If $B=B_{\G,T}$ for some graph $\G$ with chosen spanning tree $T$, then we call $B$ {\em graphical}, and denote 
$\calF_{q}(B_{\G,T})=\calF_{q}(\G,T)$, and $\calC_{q}(B_{\G,T})=\calC_{q}(\G,T)$.
\end{definition}

Note that in the case of the $q$-flow lattice, the involution $d$ is the composition of the involution $d$ on $K_{0}$ with the canonical pairing-preserving isomorphism between the $\Z[q,q^{-1}]$-submodule generated by $\{[P_{i}]: i \in E_{1}\}$, and that generated by the classes of indecomposable injective modules $\{[I_{i}]: i \in E_{1}\}$.

In the classical case, for a graph (or regular matroid) $\G$, $\calF(\G)$ and $\calC(\G)$ are mutual orthogonal complements 
in the Euclidean lattice $\Z^{{E(\G)}}$. The $q$-analogue of this statement is the following:

\begin{proposition}
 In $K_{0}$, $\mathcal{C}_q(B)$ is the right orthogonal complement to $\mathcal{F}_q(B)$, and 
 $\mathcal{F}_q(B)$ is the left orthogonal complement of 
 $\mathcal{C}_q(B)$.
\end{proposition}

\begin{proof}
Recall that simple modules form a right dual basis to indecomposable projective modules with respect to the graded Euler
pairing. That is, $\langle [P_{i}],[L_{j}]\rangle=\delta_{ij}$ for all $1\leq i,j\leq n$. This implies that the right orthogonal complement $[P_{i}]^{\perp}$ of $[P_{i}]$ is the span of $\{[L_{j}]:j\neq i\}$, since the simple modules form a basis, and 
pairing with $[P_{i}]$ on the left picks out the 
coefficient of $[L_{i}]$. 
The right orthogonal complement of $\calF_{q}(B)$ is the intersection 
$\bigcap_{e_{i}\in E_{1}}[P_{i}]^{\perp}$, that is, the span of $\{[L_{j}]: e_{j}\in E_{0}\}$, which is by definition $\calC_{q}(B)$.

The same argument, using the fact that indecomposable projective modules also form a basis, 
proves that the left orthogonal complement of $\calC_{q}(B)$ is $\calF_{q}(B)$.
\end{proof}

\begin{remark}\label{rmk:qGram}
Given a signed bipartite graph $B$, with vertex set $E=E_{0}\cup E_{1}$ it is possible to define the $q$-cut and $q$-flow lattices without reference to the category $A(B)\text{-mod}$. Namely,
in the basis $\{[P_{i}]\}_{i\in E}$, the $q$-Gram matrix of the ``Euclidean $q$-lattice'' $K_{0}$ is of the form
$$G_{K_{0},\{[P_{i}]\}}=
\left[
\begin{array}{c|c}
I & C \\
\hline
C^{t} & G_{\mathcal F}
\end{array}
\right].$$
Here $I$ is the identity matrix of size $|E_{0}|$, and $C=qM_{B}$, where $M_{B}$ denotes the signed
adjacency matrix of $B$. That is, if $i\in E_{0}$ and $j \in E_{1}$ then $C_{ij}$ is zero if $i$ and $j$ are 
not adjacent, $q$ if they are adjacent via a positive edge, and $-q$ if adjacent via a negative edge of $B$. $C^{t}$ is the 
transpose of $C$. 

The matrix $G_{\mathcal F}$ is the {\em Gram matrix of the $q$-flow lattice} in the basis $\{[P_{i}]\}_{ i \in E_{1}}$, namely, for $e_{i}, e_{j}\in E_{1}$ and $\langle \cdot, \cdot \rangle_{\calF(B)}$ denoting the pairing in the classical flow lattice of $B$,
$$(G_{\calF})_{ij}=\begin{cases}
1+(\langle C_{i},C_{i}\rangle_{\calF(B)}-1)q^{2} \quad \text{when }i=j \\
\langle C_{i},C_{j}\rangle_{\calF(B)}q^{2} \quad \text{when }i\neq j. 
\end{cases}$$

The classes $[L_{i}]$ $(i\in E)$ written in the basis $\{[P_{i}]\}_{i\in E}$ are the columns of $G_{K_{0},\{[P_{i}]\}}^{-1}$. The involution $d$ on $K_{0}$ is defined by $q$-anti-linearity and fixing the elements $\{[L_{i}]: i\in E\}$

The $q$-flow lattice can be defined as the submodule spanned by $\{[P_{i}]: e_{i}\in E_{1}\}$, with the inherited form and the involution given by $d([P_{i}])=[P_{i}]$, $d(q)=q^{-1}$.

The $q$-cut lattice is the right orthogonal 
complement of the $q$-flow lattice, or equivalently, the $q$-sublattice generated by $\{[L_{i}]:i\in E_{0}\}$.
The Gram matrix of $\calC_{q} (B)$ is given in the basis $\{[L_{i}]:i\in E_{0}\}$ is given by
$$(G_{\calC})_{ij}=\begin{cases}
1+(\langle K_{i},K_{i}\rangle_{\calC(B)}-1)q^{2} \quad \text{when }i=j \\
\langle K_{i},K_{j}\rangle_{\calF(B)}q^{2} \quad \text{when }i\neq j. 
\end{cases}$$
\end{remark}

\begin{remark}
It was an arbitrary choice to define the $q$-flow lattice via the basis of indecomposable projectives, and one equally reasonable choice would have been to use indecomposable injective modules (cf Remark~\ref{rmk:FlowInjectives}). Then $\calC_{q(B)}$ would be the {\em left} orthogonal complement of $\calF_{q}(B)$ in $K_{0}$. The Gram matrix would remain the same (but in the basis $\{[I_{i}]\}_{i\in E_{1}}$), and $d$ would fix the elements $[I_{i}].$ 
\end{remark}

In classical lattice theory, if an orthonormal basis of a $\Z$-lattice exists, then this basis is unique up to
signs and permutation of the basis vectors.  Furthermore, if $v$ is an element of norm $\pm 1$, then one of $\{v,-v\}$ is an element of this basis.  There is a similar rigidity statement for a special class of $q$-lattices which have a basis satisfying certain conditions, but which need not be orthonormal. This includes in particular $q$-cut and $q$-flow lattices:

\begin{lemma}\label{lem:rigidity}
Suppose a $q$-lattice $L$ of rank $n$ has a basis $\{B_{i}: i=1,
\dots,n\}$ and that there exists $k\in \Z$ such that for all $i,j = 1,\hdots, n$:
\begin{itemize}
\item  $\langle B_{i},B_{i} \rangle= 1+ c_{ii}q^{k}$ with $c_{ii}\in \Z$, and 
\item $\langle B_{i}, B_{j}\rangle =\langle B_{j},B_{i}\rangle= c_{ij} q^{k}$ for $i\neq j$ with
$c_{ij}\in \Z$.
\end{itemize} 
Then such a basis is unique up to permutation and signs: if for any $B\in L$, $\langle B,B \rangle=1+cq^{k}$ for some $c \in \Z$,
then $B= \pm B_{i}$ for some $i$.
\end{lemma}

\begin{proof}
Assume that for some $B\in L$, $\langle B,B \rangle=1+cq^{k}$, and $B=\sum_{i=1}^{n} \alpha_{i}B_{i}$ for some $\alpha_{i}\in \Z[q,q^{-1}]$. Then 
$$\langle B,B \rangle = \sum_{i,j=1}^{n} \bar{\alpha}_{i}\alpha_{j}\langle B_{i},B_{j}\rangle =
\left(\sum_{i=1}^{n} \bar{\alpha}_{i}\alpha_{i}\right) + q^{k}\left(\sum_{i=1}^{n} c_{ii} \bar{\alpha}_{i}\alpha_{i} +\sum_{i<j} 
c_{ij}(\bar{\alpha}_{i}\alpha_{j} + \bar{\alpha}_{j}\alpha_{i})\right),$$
where $\bar{\alpha}_{i}(q)=\alpha_{i}(q^{-1})$. 
Denote $$X(q)=\sum_{i=1}^{n} \bar{\alpha}_{i}\alpha_{i}, \quad \text{ and } \quad
Y(q)=\sum_{i=1}^{n} c_{ii} \bar{\alpha}_{i}\alpha_{i} +\sum_{i<j} 
c_{ij}(\bar{\alpha}_{i}\alpha_{j} + \bar{\alpha}_{j}\alpha_{i}).$$
Notice that $X$ and $Y$ are symmetric Laurent polynomials: $X(q)=X(q^{-1})$ and $Y(q)=Y(q^{-1})$.
So we have 
$$1+cq^{k}=X+Yq^{k}.$$
Substituting $q^{-1}$ for $q$ we obtain
$$1+cq^{-k}=X+Yq^{-k}.$$
Subtract the second equation from the first to get 
$$c(q^{k}-q^{-k})=Y(q^{k}-q^{-k}),$$
therefore $Y=c$ and $X=1$.

Going back to the definition of $X$, note that the constant term of $\bar{\alpha}_{i}\alpha_{i}$ is a non-negative integer, and
it is 0 if and only if $\alpha_{i}=0$. Since
$1=X(q)=\sum_{i=1}^{n} \bar{\alpha}_{i}\alpha_{i}$, it must be that $\alpha_{i}=0$ for all but one $i$, and for the one exception
$\alpha_{i}=\pm 1$. Hence, $B=\pm B_{i}$ for some $i$, and as a result such a basis is unique up to permutation and signs.
\end{proof}

\begin{definition} Let $L_{q}$ be a $q$-lattice with a given basis $\{B_{i}: i=1...n\}$. If for any element $x\in L_{q}$
the norm $\langle x,x\rangle$ determines whether $x=\pm B_{i}$ for some $i$, then we call $L_{q}$ {\em Gram-rigid}.
\end{definition}

In other words, Gram-rigid $q$-lattices have a ``canonical basis'' producing a Gram matrix of a certain form. The $q$-flow
lattice of a signed bipartite graph is Gram-rigid with the basis $\{[P_{i}]\}_{{i}\in E_{1}}$, and the $q$-cut lattice is Gram rigid with the basis $\{[L_{j}]\}_{{j}\in E_{0}}$. 

\subsection{The $q$-cut and $q$-flow lattices and 2-isomorphisms of graphs} \label{sec:SuWagner}

A theorem of Su--Wagner \cite{SW} and Caporaso--Viviani \cite{CV} states that the classical lattice of integer flows is
a {\em complete 2-isomorphism invariant} of 2-edge-connected graphs. That is, $\calF(\G_{1})\cong \calF(\G_{2})$ if and only if there exists a cycle-preserving
bijection between $E(\G_{1})$ and $E(\G_{2})$. Such a bijection is called a {\em two-isomorphism} of graphs. Dually, the lattice of integer cuts is a complete 2-isomorphism invariant of graphs without loops.  

The Su--Wagner result is stated in the context of regular matroids, where matroid isomorphism replaces 2-isomorphisms of graphs, and 2-edge-connected graphs translate to matroids without co-loops. This result is useful not only in graph theory and matroid theory but also in low-dimensional topology \cite{Gr}.  

The following $q$-analogue of this theorem shows that $\calF_{q}(\G,T)$ and $\calC_{q}(\G,T)$ are complete invariants of the two-isomorphism type of the pair $(\G,T)$, in other words, the $q$-lattices ``remember'' the spanning tree:

\begin{theorem}\label{thm:q-2-iso}

For 2-edge-connected loopless graphs $\G_{1}$ and $\G_{2}$ with respective spanning trees $T_{1}$ and $T_{2}$, the following are equivalent:

\begin{enumerate}

\item $\calF_{q}(\G_{1},T_{1}) \cong \calF_{q}(\G_{2},T_{2})$;

\item There exists a cycle-preserving bijection $F: E(\G_{1}) \to E(\G_{2})$ for which $F(T_{1})=T_{2}$;

\item $\calC_{q}(\G_{1},T_{1})\cong \calC_{q}(\G_{2},T_{2}).$

\end{enumerate}

\end{theorem}

\begin{proof}

We first prove $(2) \Rightarrow (1)$. It is a classical fact that the map $F$ lifts to a map of lattices $\tilde{F}: \Z^{E(\G_{1})} \to \Z^{E(\G_{2})}$, where $\tilde{F}(e)=\pm F(e)$, such that $\tilde{F}$ restricts to an isomorphism $\overline{\varphi}: \calF(\G_{1}) \xrightarrow{\cong} \calF(\G_{2})$.
Since $F(T_{1})=T_{2}$, $\overline\varphi$ sends fundamental cycles corresponding to $T_{1}$ to fundamental cycles corresponding to $T_{2}$, and therefore it 
lifts to an isomorphism $\varphi: \calF_{q}(\G_{1},T_{1}) \xrightarrow{\cong} \calF_{q}(\G_{2},T_{2})$.

To prove $(1) \Rightarrow (2)$, given an isomorphism  
$\varphi: \calF_{q}(\G_{1},T_{1}) \xrightarrow{\cong} \calF_{q}(\G_{2},T_{2}),$ set $q=1$ to obtain an isomorphism
$\overline{\varphi}: \calF(\G_{1}) \xrightarrow{\cong} \calF(G_{2})$. By a strong version of the Su--Wagner--Caporaso--Viviani
theorem (see proof of \cite[Thm.3.8]{Gr}), $\overline{\varphi}$ extends to an isomorphism $F': \Z^{E(\G_{1})} \to \Z^{E(\G_{2})},$ which 
sends each edge of $\G_{1}$ to a signed edge of $\G_{2}$.

Forgetting signs, we obtain a cycle-preserving bijection $F': E(G_{1}) \to E(G_{2}),$ which in particular sends $T_{1}$-fundamental
cycles to $T_{2}$-fundamental cycles. Note that this does not imply that $F'$ sends $T_{1}$ to $T_{2}$. If an edge $e$ of $\G_{1}$
participates in at least two $T_{1}$-fundamental cycles, then $e$ is an edge of $T_{1}$, and $F'(e)$ an edge of $T_{2}$. 
On the other hand, if an edge $e\in T_{1}$ participates in a unique fundamental cycle, say $C_{i}$, then $F'(e)$ may be 
a non-$T_{2}$ edge of $F'(C_{i})$.

However, if two edges $e_{1}, e_{2}\in E(\G_{1})$ both only appear in the fundamental cycle $C_{i}$, then $e_{1}$ and $e_{2}$ appear
together in {\em any cycle} of $\G_{1}$, and so the transposition of $e_{1}$ and $e_{2}$ is a cycle-preserving automorphism of 
$E(\G_{1})$. Hence, $F'$ can always be composed with such transpositions to obtain a cycle preserving bijection 
$F: E(\G_{1}) \to E(\G_{2})$ which sends $T_{1}$ to $T_{2}$.

To prove $(2) \Leftrightarrow (3)$, similar arguments can be made using fundamental cuts.
\end{proof}

Note that, building on the Su-Wagner result, Theorem~\ref{thm:q-2-iso} can be generalised to regular matroids with a chosen basis, and then a duality argument implies $(2) \Leftrightarrow (3)$.

The Su--Wagner/Caporaso--Viviani Theorem implies in particular that for 2-edge-connected loopless graphs the isomorphism class of $\calF(\G)$ determines the isomorphism class of $\calC(\G)$, and vice versa.   Similarly, Theorem~\ref{thm:q-2-iso} says that for these graphs with a choice of spanning tree, the isomorphism type of $\calF_{q}(\G,T)$ determines the isomorphism type of $\calC_{q}(\G,T)$ and vice versa. 
This, however, is not true in general for the $q$-cut and $q$-flow lattices associated to signed bipartite graphs whose associated matroid is non-regular: an example is shown in Figure~\ref{fig:qFlowCut}. Note that for a bipartite graph ``2-edge-connected'' means there is no isolated vertex in $V(0)$, and ``loopless'' means there's no isolated vertex in $V(1)$.

\begin{figure}[h]

\includegraphics{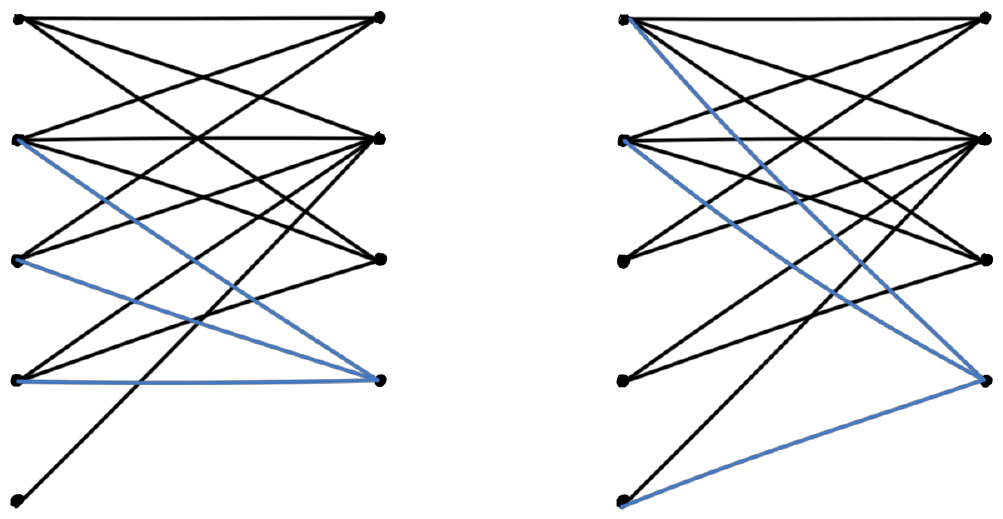}

\caption{The isomorphism class of $\calF_{q}(B)$ does not necessarily determine the isomorphism class of $\calC_{q}(B)$ for signed bipartite graphs whose associated matroid is not regular: the bipartite graphs above have identical $q$-flow lattices, but non-isomorphic $q$-cut lattices (non-isomorphism is easily seen for Gram-rigid $q$-lattices). The blue edges highlight the difference between the two graphs. }\label{fig:qFlowCut}
\end{figure}

\subsection{A $q$-Matrix-Tree theorem}

The classical integer cut and flow lattices {\em glue} together to form the Euclidean lattice $\Z^{E}$; this implies through Theorem~\ref{thm:ClassGlue} that for any signed bipartite graph $B$, $\det\calF(B) =\det\calC(B)$. Analogously, the $q$-cut and $q$-flow lattices glue to form a unimodular $q$-lattice: they embed as mutual one-sided orthogonal complements in a unimodular $q$-lattice satisfying the conditions of Theorem~\ref{thm:qGlue}.  This implies that $\det\calF_{q}(B)=\det\calC_{q}(B)$, up to units.

In graph theory, the famous Matrix-Tree Theorem states that the determinant of the classical cut and flow lattices counts the number of spanning trees for a graph. (It also applies to regular matroids, where the determinant equals the number of bases.) Our final result is a $q$-analogue of this theorem, which further illustrates how the $q$-cut and $q$-flow lattices encode information about the choice of spanning tree. This theorem also applies to regular matroids, with the same proof; here we state and prove it for graphs.

\begin{theorem}\label{thm:qMatrixTree} Given a graph $\G$ and a spanning tree $T\subseteq E(\G)$, set $r:=|T|$.  Then
$$\det \calC_{q}(\G,T)= \sum_{i=0}^{r} c_{i} q^{2i},$$ where $c_{i}$ is the number of spanning trees $T'$ of $\G$ with $|T'\cap T|=r-i$. (In particular, $c_{0}=1$.)  
\end{theorem}

This proof is based on the proof of the classical Matrix-Tree Theorem presented in \cite[Chapter 6]{Bi}.

\begin{proof}
The plan for the proof is to define a matrix $Q_{0}$, and prove that on one hand, $\det Q_{0}=\sum_{i=0}^{r} c_{i} q^{2i}$, and on the other hand, that $Q_{0}$ is a $q$-Gram matrix for $\calC_{q}(\G,T)$. 

Let $D$ denote the signed ``$q$-incidence matrix'' for $(\G,T)$: the rows of $D$ are indexed by $V(\G)$, and the columns are indexed by $E(G)$, with $T$-columns preceding non-$T$ columns. The entries of $D$ are defined as follows: 
\[D_{ij}=
\begin{cases}
1 \quad \quad \textnormal{if the edge $j \in T$ ends at the vertex $i$}, \\
-1 \quad \, \textnormal{if the edge $j \in T$ begins at the vertex $i$}, \\
q \quad \quad \textnormal{if the edge $j \notin T$ ends at the vertex $i$}, \\
-q \quad \, \textnormal{if the edge $j \notin T$ begins at the vertex $i$}. \\
\end{cases}
\] 

Define $Q:=DD^{t}$. $Q$ is a symmetric matrix whose rows and columns are indexed by $V(\G)$. By a simple computation, $Q_{ii}=\alpha_{i}+\beta_{i}q^{2}$, where $\alpha_{i}$ is the number of $T$-edges incident to vertex $i$, and $\beta_{i}$ is the number of non-$T$ edges incident to vertex $i$. When $i\neq j$, let $\gamma_{ij}$ denote the number of non-$T$ edges between the vertices $i$ and $j$, in either direction. If the vertices $i$ and $j$ are also connected by a single $T$-edge in either direction, then $Q_{ij}= -1 - \gamma_{ij}q^{2}$. Otherwise $Q_{ij}= - \gamma_{ij}q^{2}$. 

Note that the rank of $D$ is $r=|T|$, since the $T$-columns are linearly independent, and there are $r+1$ rows summing to zero. Hence, $Q$ is a symmetric $(r+1)\times (r+1)$ matrix of rank $r$, so all cofactors of $Q$ are the same. (A cofactor is the determinant of any $r\times r$ submatrix.) Let $D_{0}$ denote $D$ with the last row deleted. Let $Q_{0}:=D_{0}D_{0}^{t}$. 

For any subset $J\subseteq E(\G)$, $|J|=r$, let $D_{J}$ denote the minor of $D_{0}$ which contains the columns in $J$. By the Binet-Cauchy Theorem, $\det Q_{0}=\sum_{J} \det D_{J} \cdot \det D_{J}^{t} $. Observe that $\det D_{J}=0$ if and only if $J$ contains a cycle. 

If $J$ does not contain a cycle, then, since $|J|=r$, $J$ is a spanning tree, and $\det D_{J}=\pm q^{k}$, where $k$ is the number of non-$T$ edges in $J$. To see this, observe that there is at least one vertex $l$ which is a leaf of $J$ and does not correspond to the deleted last row of $D$. The row $l$ has a single non-zero entry in $D_{J}$, which is $\pm 1$ if the single $J$-edge incident to it is also in $T$, and $\pm q$ otherwise. Use the cofactor expansion of $\det D_{J}$ according to this row, which therefore has a single non-zero component, where the cofactor corresponds to a spanning tree for $\G \setminus l$. Hence, by repeating this process, we obtain the determinant $\pm q^{k}$.

Thus, we have shown that $\det Q_{0}=\sum_{i=0}^{r} c_{i} q^{2i}$. It remains to prove that $Q_{0}$ is a Gram matrix for $\calC_{q}(\G,T)$, by exhibiting a change of basis matrix $T$ such that the Gram matrix $(G_{\calC})_{ij}$ defined in Remark~\ref{rmk:qGram} equals $T^{t}Q_{0}T$. 

For intuition, we note that the (classical) lattice of integer cuts has a basis corresponding to any set of $r$ vertices (all but one of the vertices). The basis element corresponding to a vertex $v$ is the cut given by the partition $\{v\}\cup (V(\G)\setminus v)$. As an element of $\Z^{E(\G)}$, this cut is a linear combination of the edges incident to $v$, where incoming edges appear with $+1$ coefficient and the outgoing edges with $-1$ coefficient. The basis consists of cuts corresponding to all but one of the vertices since the cut corresponding to the last vertex is the negative sum of all the others. The lattice $Q_{0}$ is the quantised version of the Gram matrix in this vertex basis, with the last vertex omitted. The change of basis matrix $T$ is the matrix whose columns are the fundamental cuts $K_{i}$ written in the vertex basis.

With the above in mind, we define the matrix $T$ as follows. Let the edges of $T$ be numbered $1,...,r$. Label the vertices of $\G$ corresponding to the rows of $D_{0}$ with $v_{1},...,v_{r}$, in the order of their appearance as rows in $D_{0}$. Let $v_{r+1}$ denote the vertex corresponding to the omitted last row of $D$. The fundamental cut $K_{i}=V_{0,i}\cup V_{1,i}$ includes $v_{r+1}$ on one side of the partition. The change of basis matrix $T$ is the $r\times r$ matrix given by:
$$
T_{ij}=
\begin{cases}
1 & \textnormal{if $v_{i}\in V_{0,j}$ and $v_{r+1}\in V_{1,j}$; } \\
-1 &\textnormal{if $v_{i}\in V_{1,j}$ and $v_{r+1}\in V_{0,j}$; } \\
0 &\textnormal{otherwise.}
\end{cases}
$$
It is then a straightforward check that $\calC_{q}(\G,T)=T^{t}Q_{0}T=(T^{t}D_{0})(D_{0}^{t}T).$
\end{proof}

\end{document}